\documentclass{article}

\usepackage[margin=2.5cm]{geometry}
\usepackage{tikz, pgfplots}
\usepackage{comment}
\usepackage{enumitem}
\usepackage{bbm, mathtools, amsmath, amsfonts, amsthm, mathrsfs, amssymb} 
\usepackage{nicefrac}
\usepackage{nameref}
\usepackage{ifthen, etoolbox}
\usepackage[algo2e, ruled]{algorithm2e}
\usepackage{float}
\usepackage{caption}
\usepackage{subcaption}
\usepackage[round]{natbib}
\usepackage{hyperref, xcolor}
\definecolor{linkcolour}{rgb}{0,0,0.6}
\hypersetup{colorlinks,
	linkcolor=linkcolour,
	citecolor=linkcolour,
	urlcolor=magenta,
	linktocpage,
	plainpages=false}
\usepackage{cleveref}

\theoremstyle{plain}
\newtheorem{theorem}{Theorem}
\newtheorem{lemma}[theorem]{Lemma}
\newtheorem{corollary}[theorem]{Corollary}

\newtheorem{proposition}[theorem]{Proposition}
\newtheorem{example}{Example}
\newtheorem*{example*}{Example}

\theoremstyle{defintion} 
\newtheorem{definition}{Definition}
\newtheorem{assum}[definition]{Assumption}
\newtheorem*{assum*}{Assumption}

\newtheoremstyle{boldremark}
{\topsep} 		
{\topsep} 		
{\normalfont} 	
{}          	
{\bfseries} 	
{.}         	
{.5em}      	
{}          	
\theoremstyle{boldremark}
\newtheorem{remark}{Remark}
\newtheorem*{remark*}{Remark}
\newtheorem{notation}{Notation}
\newtheorem*{notation*}{Notation}

\newtheorem*{setup*}{Problem Setup}

\newcommand{\E}{\mathbb{E}}
\newcommand{\N}{\mathbb{N}}
\newcommand{\R}{\mathbb{R}}
\newcommand{\one}{{ \mathbbm{1} }}

\newcommand{\Ac}{\mathcal{A}}
\newcommand{\Fc}{\mathcal{F}}
\newcommand{\Oc}{\mathcal{O}}

\newcommand{\Oct}{\widetilde{\Oc}}

\newcommand{\Fkurrent}{\mathscr{F}}
\newcommand{\eps}{\varepsilon}

\newcommand{\nf}{{ \nabla f }}
\newcommand{\nF}{{ \nabla F }}
\newcommand{\ls}[2]{{ #1,\dots,#2 }}
\newcommand{\lsn}[1]{ \ls{{#1}_1}{{#1}_n} }
\newcommand{\al}{\alpha}

\newcommand{\Ngeq}{\N_{\geq 1}}

\newcommand{\fas}{ \forall \, }
\newcommand{\exs}{ \exists \, }

\newcommand{\alg}[1]{\texttt{#1}}

\DeclareMathOperator{\cone}{ cone }			

\newcommand{\set}[1]{{ \left\{ #1 \right\} }}
\newcommand{\pare}[1]{{ \left( #1 \right) }}

\newcommand{\norm}[1]{\left\lVert#1\right\rVert}

\newcommand{\nsq}[1]{{ \norm{#1}^2 }}

\DeclarePairedDelimiter\ceil{\lceil}{\rceil}

\newcommand{\Exp}[2][]{
	\ifthenelse{\equal{#1}{}}{\E \left[ #2 \right]}{\E \left[ #2 \, \middle| \, #1 \right]}
}
\newcommand{\Expnorm}[2][]{
	\ifthenelse{\equal{#1}{}}{\E \left[ \norm{#2} \right]}{\E \left[ \norm{#2} \, \middle| \, #1 \right]}
}
\newcommand{\Expnsq}[2][]{
	\ifthenelse{\equal{#1}{}}{\E \left[ \nsq{#2} \right]}{\E \left[ \nsq{#2} \, \middle| \, #1 \right]}
}
\newcommand{\UE}[2][]{
	\ifthenelse{\equal{#1}{}}{\E \left[ #2 \right]}{\E_{#1} \left[ #2 \right]}
}

\newcommand{\stack}[2]{{
		\stackrel{\mathclap{#1}}{#2}
		\hspace*{(\widthof{#1} - \widthof{$\mathsurround=0pt #2$})/2}
}}
\newcommand{\stackAlign}[2]{
	\hspace*{(\widthof{$\mathsurround=0pt #1$} - \widthof{$\mathsurround=0pt #2$})/2} \stackrel{\mathclap{#1}}{#2} &
	\hspace*{(\widthof{$\mathsurround=0pt #1$} - \widthof{$\mathsurround=0pt #2$})/2} 
}

\newcommand{\nameeqref}[1]{\hyperref[#1]{(\nameref{#1})}}

\newcommand{\iterationIndex}{t}

\newcommand{\fin}{1}
\newcommand{\li}{T}

\newcommand{\Lz}{L_0}
\newcommand{\Lo}{L_1}
\newcommand{\Az}{A_0}
\newcommand{\Ao}{A_1}
\newcommand{\Bz}{B_0}
\newcommand{\Bo}{B_1}
\newcommand{\Dz}{\Delta_\fin}
\newcommand{\varsym}{\sigma}
\newcommand{\tz}{\ii_0}

\newcommand{\cMeasure}{T_\eps}

\newcommand{\iterationIndexP}{{ \iterationIndex + 1 }}
\newcommand{\iterationIndexM}{{ \iterationIndex - 1 }}

\let\ii\iterationIndex
\let\ssi\stepsize

\let\iterationVector\iterationVectorNew
\newcommand{\iterationScalar}[2]{{#1}_{#2}}
\newcommand{\iterationX}[1]{\iterationVector{x}{#1}}

\newcommand{\iterationG}[1]{\iterationVector{g}{#1}}

\newcommand{\iterationXi}[1]{\iterationScalar{\xi}{#1}}

\newcommand{\genCommandsAux}[4]{
	\expandafter\def\csname#1#2\expandafter\endcsname\expandafter{\csname iteration#3\endcsname{#4}{\iterationIndex}}
	\expandafter\def\csname#1#2p\expandafter\endcsname\expandafter{\csname iteration#3\endcsname{#4}{\iterationIndexP}}
	\expandafter\def\csname#1#2m\expandafter\endcsname\expandafter{\csname iteration#3\endcsname{#4}{\iterationIndexM}}
}
\newcommand{\genCommandsName}[3]{
	\genCommandsAux{#1}{k}{#2}{#3}
	\genCommandsAux{#1}{t}{#2}{#3}
	\expandafter\newcommand\csname iteration#1\endcsname[1]{\csname iteration#2\endcsname{#3}{##1}}
}
\newcommand{\genCommands}[2]{\genCommandsName{#1}{#2}{#1}}

\genCommands{x}{Vector}
\genCommands{m}{Vector}
\genCommands{n}{Vector}
\genCommands{v}{Vector}
\genCommands{g}{Vector}
\genCommands{S}{Vector}
\genCommandsName{ga}{Vector}{\gamma}
\genCommandsName{mu}{Vector}{\mu}
\genCommands{b}{Scalar}
\genCommands{c}{Scalar}
\genCommands{C}{Scalar}
\genCommands{D}{Scalar}
\genCommands{E}{Scalar}
\genCommandsName{bt}{Scalar}{\tilde b}
\genCommandsName{xi}{Scalar}{\xi}
\genCommandsName{Fc}{Scalar}{\Fc}
\genCommandsName{proxy}{Scalar}{\proxy}
\genCommandsName{a}{Scalar}{\alpha}
\genCommandsName{d}{Scalar}{\delta}
\genCommandsName{al}{Scalar}{\eta}
\genCommandsName{eta}{Scalar}{\eta}
\genCommandsName{beta}{Scalar}{\beta}
\genCommandsName{exp}{Scalar}{\expConst}
\genCommandsName{eps}{Scalar}{\varepsilon}

\newcommand{\atxk}{\pare{\xk}}
\newcommand{\atxkp}{\pare{\xkp}}
\newcommand{\atxkm}{\pare{\xkm}}

\newcommand{\nFxk}{\nF \atxk}

\newcommand{\LzLo}{($\Lz, \Lo$)}
\newcommand{\isum}{\sum_{\ii = \fin}^\li}

\newcommand{\nsgdm}{\hyperref[alg:nsgdm]{\alg{NSGD-M}}{}}

\newcommand{\backtrGD}{\hyperref[alg:backtrGD]{\alg{GD} with Backtracking Line Search}}
\newcommand{\genmommeth}{\hyperref[alg:general_momentum_method]{{General Normalized Momentum Method}}}
\let\ix\iterationx

\newcommand{\gaj}{\iterationga j}
\newcommand{\gai}{\iterationga i}

\newcommand{\betaprod}[2]{\beta_{#1 : #2}}
\let\prodb\betaprod
\newcommand{\Aa}{{\eta \Lo}}
\newcommand{\Aat}{{\Lo \etat}}
\newcommand{\of}{{\frac 1 4}}
\newcommand{\xfin}{{\iterationx \fin}}

\newcommand{\assumlzlo}{\nameeqref{assum:lzlo_smooth}}
\newcommand{\assumlb}{\nameeqref{assum:lower_bounded}}
\newcommand{\assumbounded}{\nameeqref{assum:bounded_var}}
\newcommand{\mpp}{{-p}}
\newcommand{\mq}{{-q}}
\newcommand{\etatau}{\iterationeta \tau}
\newcommand{\xtau}{\iterationx \tau}
\newcommand{\wsum}[1]{\sum_{\ii = #1}^{b}}
\newcommand{\nFx}{{ \nF \pare x }}
\newcommand{\Acdet}{\Ac_{\det}}
\newcommand{\Fcthet}{\Fc_{\theta}}
\newcommand{\FkThet}{\Fkurrent_\Theta}
\newcommand{\FcLDz}{\Fc_{L, \Dz}}

{\let\iterationVector\iterationVectorOld}%
{\let\iterationVector\iterationVectorNew}

\title{Parameter-Agnostic Optimization under Relaxed Smoothness}
\author{Florian Hübler\footnotemark[1] \and 
		Junchi Yang\footnotemark[2] \and
		Xiang Li\footnotemark[2] \and
		Niao He\footnotemark[2]}
\date{Department of Computer Science, ETH Zurich, Switzerland}

\begin{document}

\maketitle

\renewcommand{\thefootnote}{\fnsymbol{footnote}}
\footnotetext[1]{{\texttt{fhuebler@student.ethz.ch}}}
\footnotetext[2]{{\texttt{\{junchi.yang, xiang.li, niao.he\}@inf.ethz.ch}}}
\renewcommand{\thefootnote}{\arabic{footnote}}

\begin{abstract}
	Tuning hyperparameters, such as the stepsize, presents a major challenge of training machine learning models. To address this challenge, numerous adaptive optimization algorithms have been developed that achieve near-optimal complexities, even when stepsizes are independent of problem-specific parameters, provided that the loss function is $L$-smooth. However, as the assumption is relaxed to the more realistic $(L_0, L_1)$-smoothness, all existing convergence results still necessitate tuning of the stepsize. In this study, we demonstrate that Normalized Stochastic Gradient Descent with Momentum (NSGD-M) can achieve a (nearly) rate-optimal complexity without prior knowledge of any problem parameter, though this comes at the cost of introducing an exponential term dependent on $L_1$ in the complexity. We further establish that this exponential term is inevitable to such schemes by introducing a theoretical framework of lower bounds tailored explicitly for parameter-agnostic algorithms. Interestingly, in deterministic settings, the exponential factor can be neutralized by employing Gradient Descent with a Backtracking Line Search. To the best of our knowledge, these findings represent the first parameter-agnostic convergence results under the generalized smoothness condition. Our empirical experiments further confirm our theoretical insights.
\end{abstract}

\section{Introduction}
\label{sec:intro}
We consider the unconstrained optimization problem
\begin{align}\label{eq:problam}
	\min_{x \in \R^d} F(x),
\end{align}
where $F \colon \R^d \to \R$ may be non-convex and admits access to unbiased stochastic gradients. This setting has  been extensively studied due to its prevalence in modern machine learning and data-driven optimization~\citep{Bottou2018}.

When the objective function $F$ is $L$-smooth, i.e., $F$ has $L$-Lipschitz gradients, the problem is well-explored. For the goal of finding an $\eps$-stationary point, lower bounds have been established, notably by \citet{LowerBound2}, setting a limit of $\Omega\pare{{L \Dz \varsym^2}{\eps^{-4}}}$ for stochastic first-order methods. Here $\varsym$ denotes the variance of the stochastic gradient and $\Dz$ the initialization gap. Stochastic Gradient Descent (\alg{SGD}) achieves this complexity but with stepsizes depending on problem parameters like $L$~\citep{Ghadimi2013}. Remarkably, several algorithms such as \alg{AdaGrad-Norm}, oblivious to problem parameters, are recently proven to achieve a nearly rate-optimal complexity $\Oct\pare{\eps^{-4}}$, up to the dependency on problem parameters and logarithmic factors \citep{AdaNorm_LSmooth_AffineVar, Junchi2022}. We call algorithms with this characteristic \emph{parameter-agnostic}, and \emph{parameter-dependent} otherwise.

However, \citet{Zhang20} highlight that not all machine learning applications adhere to the $L$-smoothness assumption. Their experiments in language modeling tasks revealed that the norm of the Hessian is not uniformly upper-bounded as required by $L$-smoothness. Rather, it may increase affinely with the gradient norm. To bridge the gap between theory and this observation, they introduced a more general smoothness condition termed \LzLo-smoothness:
\begin{align*}
	\norm{\nabla^2 F\pare x} \leq \Lz + \Lo \norm{\nFx}.
\end{align*}
This condition has since been further validated in various machine learning tasks~\citep{ImprovedAnalysis, ComponentWise}.

In light of this more realistic smoothness assumption, a substantial body of literature has emerged. The  nearly rate-optimal complexity $\Oct \pare{\eps^{-4}}$ has been established for various algorithms, including \alg{SGD}~\citep{li2023convex}, Clipped \alg{SGD} \citep{Zhang20, ImprovedAnalysis}, Normalized \alg{SGD} \citep{LzLoNGD}, \alg{AdaGrad-Norm} \citep{BeyondUniform, AdagradDiv} and \alg{ADAM} \citep{AdamMIT}. Yet, all of these algorithms require prior information of the problem, such as the values of $L_0$ and $L_1$. Notably, unlike the $L$-smooth setting, \alg{AdaGrad-Norm} may diverge without access to $L_1$ \citep{AdagradDiv}, shedding its fully parameter-agnostic nature. This dependence on problem parameters poses a significant challenge  as these parameters are usually unknown in practical applications, necessitating resource-intensive tuning~\citep{AdaGrad_Sharp}. These observations culminate in the pressing question:

\begin{quote}
	\textit{Is there an algorithm that converges with near-optimal complexity, without having access to any problem parameters in the ($L_0$, $L_1$)-smoothness setting?} 
\end{quote} 

With the growing interest in the development of parameter-agnostic algorithms, a fundamental \mbox{trade-off} becomes evident: while these algorithms demand less prior knowledge about the problem, they may also offer weaker convergence guarantees. For instance, under $L$-smoothness, \alg{SGD} with decaying stepsizes $\etat = \nicefrac \ssi {\sqrt \ii}$ achieves the near-optimal complexity $\Oct \pare{L\Dz\varsym^2 \eps^4}$ when $\eta$ is selected using knowledge of problem parameters~\citep{Ghadimi2013}. Without this information, however, using the same stepsizes has been shown to suffer from a lower bound of $\Omega \pare{\ssi^{-4} L^{-2}e^{\nicefrac {\pare{\ssi L}^2} 8} \eps^{-4}}$, even in the deterministic setting~\citep{Junchi2022}.  

This gap underscores the need to differentiate between parameter-agnostic and parameter-dependent algorithms when establishing lower bounds to truly grasp the potential of parameter-agnostic algorithms. However, the existing lower bound framework is constructed in a way that implicitly allows algorithms to have access to problem-specific parameters. This shortcoming raises the second key question that this paper seeks to address:
\begin{quote}
    \textit{Can we develop a lower bound framework that distinguishes between parameter-agnostic and parameter-dependent algorithms?}
\end{quote} 

\subsection{Our Contributions}
To tackle these challenges, this work makes the following contributions:

\begin{enumerate}[label=\alph*)]
	\item We show that, under the relaxed ($L_0$, $L_1$)-smoothness assumption, Normalized Stochastic Gradient Descent with Momentum (\nsgdm), as introduced by~\citep{nsgdm}, converges with a nearly rate-optimal complexity of $\Oct\pare{\eps^{-4}}$ without any prior knowledge of the problem parameters. However, it results in an exponential dependency on $\Lo$, which vanishes when the stepsize is informed by $\Lo$. Furthermore, we prove that this exponential dependency can also be avoided in the deterministic setting using Gradient Descent (\alg{GD}) with Backtracking Line Search, resulting in a complexity of $\Oc\pare{\pare{\Lz \Dz + \Lo^2 \Dz^2}\eps^{-2}}$. To the best of our knowledge, these are the first parameter-agnostic convergence results in the \LzLo-smoothness setting.
	
	\item  We provide a novel framework for lower bound analysis tailored to parameter-agnostic algorithms. Within this framework, we show that the exponential term in $\Lo$ is indispensable for a class of Normalized Momentum Methods, including \nsgdm, when the problem parameters are unknown. This framework distinctly delineates the parameter-agnostic setting from the parameter-dependent setting, in which \nsgdm~does not suffer from the exponential term. Additionally, it suggests that the \LzLo-smoothness setting may be more challenging than the $L$-smoothness setting. 
\end{enumerate}

\subsection{Related Work}
\paragraph{Parameter-Agnostic Algorithms.}
If the objective function is $L$-smooth, convergence results for \alg{SGD} are typically contingent upon stepsizes being less than $\nicefrac 2 L$~\citep{Bottou2018}. In the deterministic setting, \alg{GD} with a constant stepsize that does not satisfy this threshold may diverge~\citep{Nesterov2018}. However, this can be rectified using a Backtracking Line Search~\citep{Armijo1966, nocedal1999numerical}, which does not rely on knowing problem parameters, and achieves an optimal complexity of $\Oc(\epsilon^{-2})$. Conversely, in the stochastic setting, \citet{StochasticLinseard_Div} highlighted that line search techniques might not always converge. \alg{SGD} with a parameter-agnostic diminishing stepsize of $\nicefrac 1 {\sqrt{t}}$ still reaches a near-optimal complexity of $\Oct(\epsilon^{-4})$, though it introduces an inescapable exponential term in $L$~\citep{TwoSides}. Various adaptive methods, such as \alg{AdaGrad}~\citep{AdaGrad1, AdaGrad2}, its variants \alg{AdaGrad-Norm}~\citep{AdaGradNorm} and \alg{NSGD-M}~\citep{nsgdm}, bypass this exponential term, even without knowledge of the problem parameters, as recently shown in~\citep{AdaNorm_LSmooth_AffineVar, TwoSides}. These adaptive methods are typically considered more robust to different problem parameters~\citep{AdaGrad_Sharp, kavis2019unixgrad}, given their ability to tune algorithm hyperparameters dynamically during training. In convex optimization, certain algorithms can achieve (near-)optimal convergence rate without access to specific problem details~\citep{lan2015bundle, nesterov2015universal, levy2018online}. There is another line of research dedicated to ``parameter-free" algorithms for online convex optimization~\citep{orabona2016coin, cutkosky2018black}. However, this research emphasizes the optimal dependence on $\norm{x^*-x_0}$, where $x^*$ is the predictor in the regret bound.

\paragraph{\LzLo-Smoothness.} 
\citet{Zhang20} introduced the concept of \LzLo-smoothness, defined by the following affine bound on the Hessian-norm: $\norm{\nabla^2 F(x)} \leq L_0 + L_1 \norm{\nFx}$. The convergence of both \alg{GD} and \alg{SGD} was only recently established in this setting~\citep{li2023convex}. However, their stepsizes require prior knowledge of $\Lz$, $\Lo$, and also the exact gradient norm of the initial point, which can be unavailable in stochastic settings. Clipped \alg{SGD}~\citep{Zhang20}, and its momentum-augmented counterpart~\citep{ImprovedAnalysis}, both demand knowledge of $\Lz$ and $\Lo$ for convergence. They attain an optimal complexity of $\Oc(\eps^{-4})$ and are believed to improve over \alg{SGD} in constants. Additionally, \citet{ImprovedAnalysis} also provided a convergence result for \nsgdm{} with constant stepsizes in the appendix. Their analysis does however make use of a stronger noise assumption and requires access to all parameters. Similar complexities have been established for Normalized \alg{SGD} \citep{LzLoNGD}, signed \alg{SGD}~\citep{ComponentWise}, \alg{AdaGrad-Norm} \citep{BeyondUniform, AdagradDiv}, and \alg{ADAM} \citep{AdamMIT, ProvableAdaptivity}. However, each of these methods requires prior knowledge of problem-specific parameters. Notably, in stark contrast to the $L$-smooth setting, even \alg{AdaGrad-Norm} is not wholly parameter-agnostic. It risks divergence if the stepsize is not informed by $\Lo$, despite the method generally demanding knowledge of fewer problem parameters than other algorithms~\citep{AdagradDiv}.

\paragraph{Lower Bound Theory.} Lower bounds for seeking near-stationary points have been extensively studied within the $L$-smoothness setting. \cite{Nesterov2012} first addressed constrained optimization under box constraints. Subsequently, a seminal study by \cite{LowerBound1} established a tight lower bound of $\Omega \pare{\Dz L \eps^{-2}}$ for the deterministic setting. \cite{LowerBound2} extended the results to the stochastic setting, introducing the $\Omega \pare{\Dz L \varsym^2 \eps^{-4}}$ lower bound. Specific algorithms, such as \alg{SGD}~\citep{drori2020complexity} and \alg{Newton's method}~\citep{cartis2010complexity}, also have associated lower bounds. However, the algorithm classes considered by these lower bounds include algorithms with stepsizes that can depend on problem parameters, so they might not be tight in the parameter-agnostic setting. 
\citet{StochasticLinseard_Div}  discovered that parameter-agnostic SGD with a specific exponentially decreasing stepsize suffers from an exponential dependence during its initial phase when minimizing strongly convex functions.
Later, \citet{TwoSides} also derived a lower bound for \alg{SGD} under a polynomially decreasing stepsize in the nonconvex setting. Yet the implications of the parameter-agnostic lower bound for a class of algorithms remain ambiguous. The aforementioned studies consider the function class of $L$-smooth functions, so they are also applicable to \LzLo-smooth functions. In the realm of online convex optimization, \citet{cutkosky2016online, cutkosky2017online} have introduced a lower bound featuring an exponential term when the norm of the predictor and the Lipschitz constant are allowed to scale with the total number of iterations.

\section{Preliminaries}
\label{sec:prelims}
Let us introduce basic notations, definitions and assumptions needed in the upcoming analysis.

\begin{notation*}
	Throughout the paper, $d\in\Ngeq$ denotes the dimension of the variable to be optimized, $F\colon \R^d \to \R$ the objective and $\nf \pare{\cdot, \cdot}$ the gradient oracle. We use the common convention that empty sums and products are given by their corresponding neutral element. The conic combination of $\lsn x \in \R^d$ is denoted by $\cone\pare{\lsn x} \coloneqq \set{\sum_{i = 1}^n \lambda_i x_i \colon \lsn \lambda \geq 0}$. Functions that formally are not defined in a certain point, but have an continuous extension are to be understood as their extension. In particular, $\frac{e^x - 1}{x}$ has the value $1$ in $x=0$.
\end{notation*}

\begin{setup*}
	Since finding a solution to \eqref{eq:problam} is computationally intractable~\citep{nonconvex_NP}, we aim to find an $\eps$-stationary point. Furthermore we only allow access to a (possibly noisy) gradient oracle $\nf \pare{\cdot, \xi}$ of $\nF$, where $\xi$ is a random vector. Due to this randomness, our specific goal is finding an approximate, random solution $x \in \R^d$ with $\Expnorm{\nFx}\leq \eps$. 
\end{setup*}

Building on established work in stochastic optimization \citep{Ghadimi2013,LowerBound2}, we employ the following two de-facto standard assumptions in various results of this study.

\begin{assum}[Lower Boundedness]\label{assum:lower_bounded}
	The objective function $F$ is lower bounded by $F^* > -\infty$.
\end{assum}

\begin{assum}[Bounded Variance]\label{assum:bounded_var}
	The gradient oracle is unbiased and has finite variance, i.e.~there exists $\varsym \geq 0$ such that 
	\begin{enumerate}[label=\roman*),topsep=0pt]
		\item $\Exp{\nf \pare{x, \xi}} = \nFx$, and
		\item $\Expnsq{\nf \pare{x, \xi} - \nF \pare x} \leq \varsym^2$.
	\end{enumerate}
\end{assum}

Instead of the traditional $L$-smoothness assumption, we adopt the weaker concept of \LzLo-smoothness, as proposed by \citet{Zhang20}. Following the work of \citet{ImprovedAnalysis}, we choose a definition that does not require the Hessian. This definition is therefore weaker than the original \LzLo-smoothness assumption by \citet[Definition 1]{Zhang20}.
\begin{definition}\label{def:lzlo_smoothness}
	Let $\Lz, \Lo \geq 0$ and $f \colon \R^d \to \R$ be a differentiable function. Then $f$ is called \emph{\LzLo-smooth} if for all $x,y \in \R^d$ and all $c > 0$ with $ L_1 \norm{x-y} \leq c$ it holds that
	\begin{align*}
		\norm{\nf \pare x - \nf \pare y}
		\leq  \pare{\Az \pare c \Lz + \Ao \pare c \Lo \norm{\nf \pare x}} \norm{x-y},
	\end{align*}
	where $\Az\pare c \coloneqq 1 + e^c - \frac{e^c - 1}{c}$ and $\Ao \pare c \coloneqq \frac{e^c - 1}{c}$.
\end{definition}

\begin{assum}[\LzLo-smoothness]\label{assum:lzlo_smooth}
	The objective function $F$ is \LzLo-smooth.
\end{assum}

Notably, the two definitions are equivalent if the objective function is twice differentiable as the following lemma shows. The proof can be found in \Cref{sec:app.basic_properties}.

\begin{lemma}\label{lem:equivalence_different_lzlo_defs}
	Let $F \colon \R^d \to \R$ be twice continuously differentiable and $\Lz, \Lo \geq 0$. Then $F$ satisfies $\norm{\nabla^2 F(x)} \leq \Lz + \Lo \norm{\nFx}$ if and only if $F$ is \LzLo-smooth according to \Cref{def:lzlo_smoothness}.
\end{lemma}

\section{Parameter-Agnostic Upper Bounds}
\label{sec:upper_bounds}
In this section, we present the first parameter-agnostic convergence results on \LzLo-smooth functions. In \Cref{sec:upper_bounds.stochastic}, we show that in the stochastic setting,  \nsgdm~(see \Cref{alg:nsgdm}) achieves the nearly rate-optimal complexity of $\Oct\pare{\eps^{-4}}$, even without access to problem-dependent parameters. However, this is accompanied by an undesirable exponential dependence on $\Lo$. In \Cref{sec:upper_bounds.deterministic} we show that in the deterministic setting, \backtrGD~can avoid this exponential dependence, while still being parameter-agnostic.

\subsection					{Stochastic Setting}
\label{sec:upper_bounds.stochastic}
The convergence of \nsgdm{} occurs in two phases. In the initial adaptation phase, the algorithm accumulates error due to a large stepsize. Unfortunately, this error grows exponentially with $\Lo$. This behaviour is intrinsic to \nsgdm~and cannot be eliminated, as we will show in \Cref{sec:lower_bounds}. Once the stepsize decreases below a threshold (which is polynomial in $\Lo$), the algorithm transitions into a convergence phase. In this latter phase, the error decays at a rate of $\li^{- \nicefrac 1 4} \log \pare \li$. The following \Cref{thm:main_result} formalizes this behaviour, and the proof can be found in \Cref{sec:app.missing_proofs.upper_bounds.stochastic_setting}.

\begin{algorithm2e}[t]
	\caption{Normalized \alg{SGD} with Momentum (\alg{NSGD-M}) \citep{nsgdm}}
	\label{alg:nsgdm}
	\DontPrintSemicolon
	\SetKwInOut{Input}{Input}
	\Input{Starting point $\iterationx \fin \in \R^d$, stepsizes $\etat > 0$, moving average parameters $\betat \in [0,1)$}
	$\iterationm 0 \gets 0$\;
	\For{$\iterationIndex = 1, 2, \,\hdots$}{
		Indep.~sample $\xit$ from the distribution of $\xi$.\;
		$\gk \gets \nf \pare{\xk, \xit}$\;
		$\mk \gets \betat \mtm + \pare{1 - \betat}\gk$\;
		$\xkp \gets \xk - \frac \etat {\norm{\mk}} \mk$\;
	}
\end{algorithm2e}

\begin{theorem}[Convergence of \nsgdm]\label{thm:main_result}
	Assume \assumlb, \assumlzlo~and \assumbounded. Furthermore, define the parameters $\betat \coloneqq 1 - \ii^{-\nicefrac 1 2}$ and $\etat \coloneqq \frac {\ii^{-\nicefrac 3 4}} 7$.
	Then \nsgdm~with starting point $\iterationx \fin \in \R^d$ satisfies
	\begin{align*}
		\frac 1 \li \isum \Expnorm \nFxk
		\leq&\ \Oct\pare{\frac{\Dz e^{\Lo^2} + \varsym + e^{\Lo^2} \Lz}{\li^{\nicefrac 1 4}}},
	\end{align*}
	where $\Dz \coloneqq F \pare{\xfin} - F^*$ is the initialization gap.
\end{theorem}

Since \LzLo-smoothness includes $L$-smoothness as a special case, the lower bound of $\Oc(\eps^{-4})$ to find an $\eps$-stationary point is still applicable here. \Cref{thm:main_result} implies an optimal complexity in $\eps$ up to the logarithmic factor without any prior knowledge of the problem parameters, but it comes with the cost of an exponential term in $L_1$. The following corollary shows that this cost arises from the parameter-agnostic stepsize; that is, the exponential term disappears when the stepsize is informed by the value of $\Lo$ only.

\begin{corollary}\label{cor:non_agnostic_nsgdm}
Under the assumptions of \Cref{thm:main_result}, define the parameters $\betat \coloneqq 1 - \ii^{-\nicefrac 1 2}$ and $\etat \coloneqq \frac {\ii^{-\nicefrac 3 4}} {12\Lo}$.
Then \nsgdm~with starting point $\iterationx \fin \in \R^d$ satisfies
\begin{align*}
	\frac 1 \li \isum \Expnorm \nFxk
	\leq \Oct\pare{\frac{ \Lo \Dz + \varsym + \frac{\Lz} {\Lo}}{\li^{\nicefrac 1 4}}},
\end{align*}
where $\Dz \coloneqq F \pare{\xfin} - F^*$ is the initialization gap.
\end{corollary}

These results indicate that \nsgdm~is potentially more robust to hyper-parameter selection than other existing algorithms. In comparison, \alg{SGD} necessitates knowledge of both $L_0$ and $L_1$, as well as the exact value of $\|\nabla f(x_1)\|$~\citep{li2023convex}. Clipped \alg{SGD} requires to know $L_0$ and $L_1$~\citep{ImprovedAnalysis}, and even \alg{AdaGrad-Norm} demands knowledge of $L_1$~\citep{BeyondUniform, AdagradDiv}. It is important to note that our analysis is significantly different from the previous analysis for \nsgdm{} in \citep{ImprovedAnalysis}. The latter focused on constant stepsizes and momentum parameters determined by $L_0$, $L_1$, target accuracy $\eps$, and variance $\sigma$. It furthermore made use of a stronger noise assumption.

\subsection					{Deterministic Setting}
\label{sec:upper_bounds.deterministic}
Given the prior results, one might naturally wonder if there exists any algorithm that can attain parameter-agnostic convergence without exponential dependence on $\Lo$. The subsequent theorem confirms that this is indeed possible, at least in the deterministic setting. This is achieved by using Gradient Descent with a Backtracking Line-search (see \Cref{alg:backtrGD}).

\begin{algorithm2e}[t]
	\caption{\alg{GD} with Backtracking Line Search}
	\label{alg:backtrGD}
	\DontPrintSemicolon
	\SetKwInOut{Input}{Input}
	\Input{Starting point $\iterationx \fin \in \R^d $, Armijo parameters $\beta \in (0,1)$ and $\gamma \in (0,1)$}
	$\iterationeta 0 \gets 1$\;
	\For{$\iterationIndex = 1, 2, \,\hdots$}{
		Choose $k \in \N$ minimal such that $\beta^k \leq \etatm$ and
        $F\pare{\xt - \beta^k \nFxk} \leq F \atxk - \beta^k \gamma \nsq{\nFxk}$\;
		$\etat \gets \beta^k$\;
		$\xkp \gets \xk - \etat \nFxk$\;
	}
\end{algorithm2e}

\begin{theorem}\label{thm:linesearch_result}
	Assume \assumlb~ and \assumlzlo~ in the deterministic setting. Then \backtrGD~(see \Cref{alg:backtrGD}) with parameters $\beta, \gamma \in (0,1)$ satisfies
	\begin{align*}
		\frac 1 \li \isum \nsq \nFxk 
		\leq \frac{4 \Lz \Dz + 14 \Lo^2 \Dz^2}{\beta \gamma \pare{1-\gamma} \li},
	\end{align*}
	where $\Delta_\fin \coloneqq F \pare{\ix \fin} - F^*$. 
\end{theorem}

This implies a complexity of $\Oc \pare{\pare{\Lz \Dz + \Lo^2 \Dz^2}\eps^{-2}}$, which is optimal in the dependence of $\eps$ and $L_0$ in the deterministic setting. The proof rests on the observation that \backtrGD~is a descent algorithm and hence both the function value and gradient norm remain upper bounded along the trajectory. Consequently, the algorithm behaves as if it is addressing $\pare{\Lz + \Lo C}$-smooth functions, where $C$ represents the gradient norm's upper bound. The formal proof can be found in \Cref{sec:app.missing_proofs.upper_bounds.deterministic_setting}. We have not extended our considerations to the stochastic setting for this algorithm, as a stochastic line search can potentially fail even under the stricter $L$-smoothness assumption~\citep{StochasticLinseard_Div}.

\section{Parameter-Agnostic Lower Bounds}
\label{sec:lower_bounds}
In the previous section, we highlighted that the first provable parameter-agnostic algorithm, \nsgdm, comes at the cost of an exponential $\Lo$-dependence. This naturally raises the question: Is such an undesirable term unavoidable for this class of algorithms? Since most existing lower bounds focus on the parameter-dependent setting --- where hyper-parameters of algorithms can be set based on problem parameters --- we begin by introducing the concept of lower bounds specifically designed for parameter-agnostic setting in \Cref{sec:lower_bounds.intro}. Subsequently, in \Cref{sec:lower_bounds.lb}, we utilize this concept to show that \nsgdm{} indeed suffers from an exponential dependence on $\Lo$.

\subsection					{A Lower Bound Framework}
\label{sec:lower_bounds.intro}
To motivate the need for specific lower bounds for parameter-agnostic algorithms, let us consider the algorithm class $\Ac$ consisting of \alg{GD} with all constant stepsizes $\{\ssi: \ssi > 0\}$. Furthermore, we consider the well-studied function class comprising of $L$-smooth functions with initialization gap $F \pare \xfin - F^* \leq \Dz$, denoted as $\Fc_{L, \Dz}$. A well-established lower bound for $\Ac$ to find an $\epsilon$-stationary point in this setting is $\Omega \pare{L\Dz \eps^{-2}}$, as demonstrated by the seminal work of \citet{LowerBound1}. This lower bound is tightly matched by \alg{GD} with a parameter-dependent stepsize smaller than $\nicefrac 2 {L}$, and hence cannot be improved. However, without knowledge of $L$, \alg{GD} with constant stepsize generally fails to converge \citep{Nesterov2018}. Thus, a parameter-agnostic notion of lower bounds would be more informative under this setting.

For simplicity, our discussion focuses on deterministic algorithms. However, this can be readily generalized to the stochastic setting by incorporating a stochastic oracle into the algorithm's definition, as detailed in~\citep{LowerBound2}. Additionally, algorithms with a deterministic gradient oracle can be viewed as specific instances of their stochastic equivalents when there is no gradient noise. Consequently, the lower bounds established for deterministic algorithms are also applicable to their stochastic counterparts.

\begin{definition}[Deterministic Algorithm~\citep{LowerBound1}]
	We say that $A$ is a \emph{first order deterministic algorithm} if it, given a differentiable function $f$, produces iterates of the form
	\begin{align*}
		\xt = A_\ii \pare{f \pare{\xfin}, \nf \pare{\xfin}, \dots, f \atxkm, \nf \pare{\xtm}},
	\end{align*}
	where $A_\ii$ is a (Lebesgue-) measurable mapping. We denote the set of all such algorithms as $\Ac_{\det}$.
\end{definition}

It is important to note that an algorithm $A \in \Acdet$ is \emph{a function that takes a differentiable function as its argument and outputs a sequence} in $\R^d$. When we mention an  ``algorithm with a specific stepsize scheme'', we are technically referring to a set of algorithms $\set{A_\ssi}$, where $\ssi$ serves as the hyperparameter of this stepsize scheme,  and each distinct $\ssi$ defines a unique algorithm.

In the following, we consider general parametrized function spaces, denoted by $\Fcthet$, where the parameter $\theta$ resides in a parameter space $\Theta \subseteq \R^k$. Specifically, in \Cref{sec:lower_bounds.lb}, we will use this concept with $\theta = (\Dz, \Lz, \Lo)$. We use $\Fkurrent_\Theta \coloneqq \set{\Fcthet \colon \theta \in \Theta}$ to denote a parametrized family of function spaces. For an algorithm class $\Ac \subseteq \Acdet$, existing lower bound literature usually considers the following challenge~\citep{LowerBound1}:
for any problem parameter $\theta \in \Theta$ and target accuracy $\epsilon > 0$, find a lower bound for 
\begin{align} \label{eq:conventional_lb}
  \inf_{A \in \Ac} \sup_{f \in \Fc_{\theta}} \cMeasure\pare{A, f},
\end{align}
where $\cMeasure\pare{A, f} \coloneqq \inf \set{\ii \in \Ngeq \mid \norm{\nf \atxk} \leq \eps}$ denotes the number of iterations required for $A$ to reach an $\eps$-stationary point of $f$. Importantly, the parameter $\theta$ is fixed for the function class $\Fc_{\theta}$ before taking the infimum over the algorithm class and supremum over the function class. This implicitly allows algorithms to adjust their hyperparameters based on $\theta$.

In contrast to the framework above, we propose the concept of \emph{parameter-agnostic lower bounds}. 
\begin{definition}[Parameter-Agnostic Lower Bound] \label{def:lb}
	Let $\Ac \subseteq \Acdet$ be an algorithm class and $\Fkurrent_\Theta \coloneqq \set{\Fcthet \colon \theta \in \Theta}$ a parameterized family of function spaces.
  	A function $g\colon (0, \infty) \times \Theta \to [0, \infty]$ is called a \emph{parameter-agnostic lower bound of $\Ac$ on $\Fkurrent_\Theta$} if there does not exist an algorithm $A \in \Ac$ such that for all $\kappa > 0$, there exists $\eps_0, \theta_0 > 0$ such that for all $\eps \leq \eps_0, \theta \geq \theta_0$, $\sup_{f \in \Fcthet} \cMeasure\pare{A, f} \leq \kappa g\pare{\eps, \theta}$. The comparisons of $\theta \in \Theta$ and scalars are to be understood component-wise.
\end{definition}

In summary, the definition states that $g\pare{\eps, \theta}$ serves as a parameter-agnostic lower bound if no algorithm has a complexity that is ``asymptotically better than'' $g$. The performance of $A$ is therefore evaluated across all function spaces with $\theta$ large enough. This relation is formalized and elaborated upon in \Cref{sec:app.additional_discussion_lower_bounds}. 
The conventional definition in \Cref{eq:conventional_lb} on the other hand states that for any $\theta \in \Theta$ --- which means the parameter is determined first --- there does not exist an algorithm $A \in \Ac$ such that for all $\kappa > 0$, $\sup_{f \in \Fcthet} \cMeasure\pare{A, f} \leq \kappa g\pare{\eps, \theta}$. This implies that the the algorithm with the optimal hyperparameters for this $\theta$ is included in the assessment.
Finally it should be noted that earlier parameter-agnostic lower bounds presented in \citep{TwoSides, StochasticLinseard_Div} apply solely to a particular algorithm with a specified stepsize. Consequently, it is ambiguous how one might define a lower bound across a class of algorithms.

Additionally, \Cref{def:lb} also outlines a way to compare different algorithms within the parameter-agnostic framework by assessing their performance with asymptotically large $\theta$. We dive deeper into this possibility in \Cref{sec:app.additional_discussion_lower_bounds.another_view}. This approach also facilitates a simplified technique for establishing $g$ as an parameter-agnostic lower bound of $\Ac$ on $\Fkurrent_\Theta$.
\begin{proposition} \label{prop:sufficient_lb}
  If, for any algorithm $A \in \Ac$, there exist constants $\eps_0, \theta_0, K > 0$ such that
\begin{align*}
	\fas \eps \in (0, \eps_0], \theta \geq \theta_0 \colon \sup_{f \in \Fcthet} \cMeasure\pare{A, f} \geq K g \pare{\eps, \theta},
\end{align*}
where $\theta \in \Theta$, then $g$ is a parameter-agnostic lower bound of $\Ac$ on $\Fkurrent_\Theta$.
\end{proposition}
The condition outlined in \Cref{prop:sufficient_lb} is particularly convenient for constructing lower bounds, and will be our primary tool for that purpose in the upcoming subsection. However, \Cref{def:lb} offers a more precise depiction of the asymptotic behaviours exhibited by the algorithm class, compared to the condition in \Cref{prop:sufficient_lb}. A more comprehensive discussion of the distinctions between these two formulations can be found in \Cref{sec:app.additional_discussion_lower_bounds.another_view}.

The upcoming example demonstrates how the notion of parameter-agnostic lower bounds is able to close the gap described in the beginning of this section.

\begin{example} [Parameter-Agnostic Lower Bound for Constant Stepsize \alg{GD}]
  In the parameter-dependent regime, \alg{GD} with properly tuned constant stepsize converges for $L$-smooth functions. However, it is well-known that \alg{GD} with stepsize $\eta > \nicefrac 2 L$ does not converge in general. We now show how this is reflected by our framework of parameter-agnostic lower bounds.
  Let $A_\eta$ denote \alg{GD} with constant stepsize $\eta > 0$ and starting point $\xfin = 0$. Furthermore let $\FcLDz$ be the set of $L$-smooth functions with initialization gap $F \pare{0} - \inf_{x \in \R^d} F(x) \leq \Dz$. It is well-known that $A_\eta$ diverges on the function $F(x) = \frac L 2 (x + \delta)^2$ if $L > \nicefrac 2 \ssi$ and $\norm{\nF \pare{\xfin}} > \eps$. In particular, by setting $\delta = \sqrt{\nicefrac {2\Dz} {L}}$, we have for all $L, \Dz \geq \theta_0 \coloneqq \max\set{1, \nicefrac 3 \ssi}$ and $\eps \leq \eps_0 \coloneqq 1$ that $\sup_{f \in \FcLDz}\cMeasure\pare{A_\eta, f} = \infty$. By choosing $K=1$ and $\theta_0, \eps_0$ as above we hence obtain
  \begin{align*}
      \fas \eps \in (0, \eps_0], L, \Dz \geq \theta_0 \colon \sup_{f\in\FcLDz} \cMeasure\pare{A_\eta, f} \geq 1 \cdot \infty.
  \end{align*}
 Since we chose $\ssi$ arbitrary in the start, 
 \Cref{prop:sufficient_lb} implies that $g \equiv \infty$ is a parameter-agnostic lower bound for the family of \alg{GD} with all constant stepsizes on $\FcLDz$.
\end{example}

\subsection		{Lower Bound for A Family of Normalized Momentum Methods}
\label{sec:lower_bounds.lb}
In this subsection we establish a parameter-agnostic lower bound for a generalized version of \nsgdm. More specifically, for $\ssi > 0$ and $\al \in (0,1)$, we consider the following iteration rule:
\begin{align}\label{eq:general_momentum_method}
\begin{split}
    &\gt \gets \nf \pare{\xt, \xit}\\
    &\text{Choose } \mt \in \cone \pare{\ls{\iterationg \fin}{\gt}}\\
    &\xtp \gets \xt - \frac{\ssi}{\ii^\al} \frac {\mt} {\norm{\mt}}
\end{split}
\end{align}
We call algorithms that follow this procedure \emph{General Normalized Momentum Methods} (see also \Cref{alg:general_momentum_method} in \Cref{sec:app.missing_proofs.lower_bounds}). It is clear that \nsgdm{} from \Cref{thm:main_result} is a member of this family of algorithms.

\begin{theorem}[Parameter Agnostic Lower Bound for \genmommeth\textcolor{linkcolour}{s}]\label{thm:param_agnostic_lb}
    Let $\Ac$ be the class of algorithms defined by \eqref{eq:general_momentum_method} with $\al \geq \nicefrac 1 2$ and $\Fc_{\Lz, \Lo, \Dz}$ the set of \LzLo-smooth functions with $F \pare{x_1} - \inf_{x \in \R^d} F(x) \leq \Dz$. Furthermore assume the deterministic setting, i.e.\ the gradient oracle returns the true gradient. Then, for any $\delta \in (0,1)$,
	\begin{align*}
		g \pare{\eps, \theta} \coloneqq \pare{\frac{\Dz + \frac{1}{\Lo} \pare{e^{\Lo^{1-\delta}} - 1}}{\eps}}^2
	\end{align*}
	is a parameter-agnostic lower bound for $\Ac$ on $\set{\Fc_{\Lz, \Lo, \Dz} \colon \Lz, \Lo, \Dz \geq 0}$. The subset $\tilde \Ac \subseteq \Ac$ that satisfies $\al \geq \nicefrac 3 4$ furthermore has the parameter-agnostic lower bound
	\begin{align*}
		\tilde g \pare{\eps, \theta} \coloneqq \pare{\frac{\Dz + \frac{1}{\Lo} \pare{e^{\Lo^{1-\delta}} - 1}}{\eps}}^4.
	\end{align*}
	In particular this lower bound applies to \nsgdm{} in \Cref{thm:main_result}.
\end{theorem}

This lower bound reveals that one cannot achieve a parameter-agnostic convergence result for \nsgdm{} without an exponential dependence on $\Lo$. It is important to note that the above is a \emph{parameter-agnostic lower bound}. Consequently, this finding does not contradict \Cref{cor:non_agnostic_nsgdm}.
Moreover, it also suggests that finding an $\eps$-stationary point in a parameter-agnostic fashion is strictly harder in this relaxed smoothness setting: in the $L$-smooth setting, equivalent to $(L, 0)$-smoothness, the exponential term in \Cref{thm:main_result} vanishes, aligning with previous upper bounds~\citep{TwoSides, nsgdm}.

\paragraph{Proof Sketch} 
As shown in \Cref{prop:sufficient_lb}, to establish such a lower bound, we need a set of hard functions for each algorithm and large enough parameters. The subsequent Lemma accomplishes this.

\begin{lemma}\label{lem:momentumlb_functions}
  Consider a \genmommeth~${A}$ with parameters $\ssi > 0$ and $\al \in (0,1)$. Let $0 < \eps < \nicefrac 1 2, \Dz \geq \nicefrac 1 4, \Lz \geq \nicefrac 8 \ssi$ and $\Lo > 0$. Then there exists an \LzLo-smooth function $F$ and initialization $x_1$ with $F \pare{x_1} - \inf_{x \in \R^d} F(x) \leq \Dz$ for which ${A}$ requires at least
	\begin{align*}
		\li \geq \pare{\frac{1-\al}{2}}^{\frac 1 {1 - \al}} \pare{\frac{\Dz} \ssi + \frac 2 \Aa \pare{e^{\frac \Aa 4} - 1}}^{\frac 1 {1 - \al}}\eps^{-\frac 1 {1 - \al}}
	\end{align*}
	iterations to find an $\eps$-stationary point in the deterministic setting.
\end{lemma}

\newcommand{\maxVal}{M}
\begin{figure}[t]
	\centering
	\begin{subfigure}[c]{0.4\textwidth}
		\begin{tikzpicture}
	\pgfmathsetmacro{\varLz}{20}
	\pgfmathsetmacro{\varLo}{20}
	\pgfmathsetmacro{\varEps}{0.25}
	\pgfmathsetmacro{\varDz}{0.5}
	\pgfmathsetmacro{\varSsi}{0.6}
	\pgfmathsetmacro{\varZOne}{2/\varLz}
	\pgfmathsetmacro{\varZTwo}{\varSsi - (1 - 2*\varEps)/\varLz}
	\pgfmathsetmacro{\a}{0}
	\pgfmathsetmacro{\b}{\varZOne}
	\pgfmathsetmacro{\c}{\varSsi / 2}
	\pgfmathsetmacro{\d}{\varSsi - \varZOne}
	\pgfmathsetmacro{\e}{\varZTwo}
	\pgfmathsetmacro{\expTerm}{2*(exp(\varLo*(\c - \b))-1)/\varLo}
	\begin{axis}[
		xmin=-0.5, xmax=1.2,
		ymin=0, ymax=7,
		axis lines=middle,
		xlabel={\( x \)},
		ylabel={\( F(x) \)},
		samples=200,
		domain=-1:1.2,
		height=5.5cm,
		width=7cm,
		ytick={0,0.5,5.86},
		xtick={0, \c, \varZOne, \e},
		xticklabels={$0$, $\frac \ssi 2$, $z_1$, $z_2$},
		extra x ticks={\varSsi},
		extra x tick labels={$\ssi$},
		extra x tick style={
			xticklabel style={yshift=6mm, xshift=0.5mm}
		},
		yticklabels={$0$,$\Dz$,$\maxVal$},
		]
		
		\addplot[black, thick, domain=-1:\a] 
		{\varDz - x};
		
		\addplot[black, thick, domain=\a:\b] 
		{\varDz + \varLz*x^2/2 - x};
		
		\addplot[black, thick, domain=\b:\c] 
		{\varDz + (exp(\varLo*(x - \b))-1)/\varLo};
		
		\addplot[black, thick, domain=\c:\d] 
		{\varDz + \expTerm - (exp(\varLo*(\d-x))-1)/\varLo};
		
		\addplot[black, thick, domain=\d:\e] 
		{\varDz + \expTerm - \varLz*(x-\d)^2/2 + (x-\d)};
		
		\addplot[black, thick, domain=\e:1.2] 
		{\varDz + \expTerm - \varLz*(\e-\d)^2/2 + (\e - \d) - 2*\varEps*(x-\e)};
		
	\end{axis}
\end{tikzpicture}
		\caption{Plot of $F(x)$}
		\label{fig:lower_bound_F}
	\end{subfigure}
	\hfill
	\begin{subfigure}[c]{0.55\textwidth}
		\begin{tikzpicture}
	\begin{axis}[
		domain=-0.1:1.55,
		samples=200,
		axis lines=middle,
		xlabel={$x$},
		ylabel={$F'(x)$},
		height=6cm,
		width=10cm,
		ytick={0,9},
		xtick={0,0.125,0.5},
	xticklabels={$0$,$z_1$,$\frac{\ssi}{2}$,$z_2$,$z_3$},
	extra x ticks={0.97, 1.4},
	extra x tick labels={$z_2$, $z_3$},
	extra x tick style={
		xticklabel style={yshift=-5pt}
	},
	yticklabels={$0$,$e^{\nicefrac \Aa 4}$},
	xmin=-0.2, xmax=1.55, 
	ymin=-2, ymax=10  
	]
	
	\addplot[black, thick, dotted, domain=-0.15:-0.05] {-1};  
	
	\addplot[black, thick, domain=-0.05:0] {-1};  
	
	\addplot[black, thick, domain=0:0.125] {16*x - 1};
	
	\addplot[black, thick, domain=0.125:0.5] {e^(6*(x - 0.125))};
	
	\addplot[black, thick, domain=0.5:0.875] {e^(6*(0.875 - x))};
	
	\addplot[black, thick, domain=0.875:0.97] {16*(1-x) - 1};
	
	\addplot[black, thick, domain=0.97:1.1] {-0.52};
	
	\draw [black, thick, dotted] (axis cs: 1.1, -0.52) -- (axis cs: 1.3, -0.52);
	
	\addplot[black, thick, domain=1.3:1.4] {-0.52};
	
	\addplot[black, thick, domain=1.4:1.4325] {16*(x-1.4) - 0.52};
	
\end{axis}
\end{tikzpicture}
		\caption{Plot of $F'(x)$}
		\label{fig:lower_bound_Fp}
	\end{subfigure}
	\caption{Plot of the hard function used in the proof of \Cref{lem:momentumlb_functions}.}
	\label{fig:lower_bound}
\end{figure}
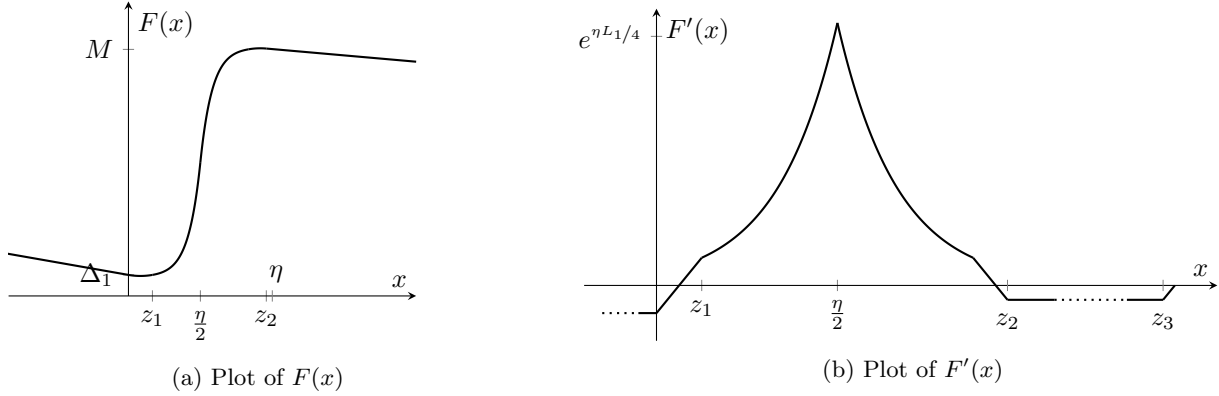

\begin{proof}[Proof of \Cref{lem:momentumlb_functions}]
To prove the lemma, we consider the following function constructed by its derivatives:
$F(x) \coloneqq \Dz + \int_0^x F' d\lambda$, with
\begin{align*}
  F'(x) = \begin{cases}
    -1, &\text{if } x \leq 0 \\
    \Lz x - 1,&\text{if } 0 < x \leq z_1 \\
    e^{\Lo \pare{x - z_1}},&\text{if } z_1 < x \leq \frac{\eta}{2} \\
    F'(\eta - x),&\text{if } \frac{\eta}{2} < x \leq z_2 \\
    -2\eps,&\text{if }  z_2 < x \leq z_3, \\
    \frac{2\eps}{z_4 - z_3}x - \frac{2\eps z_3}{z_4 - z_3} - 2\eps,&\text{if }  z_3 < x \leq z_4, \\
    0,&\text{if }  x > z_4 \\
  \end{cases},
\end{align*}
where
\begin{align*}
  z_1 \coloneqq \frac{2}{L_0}, \quad z_2 \coloneqq \eta - \frac{1-2\epsilon}{L_0},
  \quad z_3 \coloneqq \eta + \frac{M}{2\eps},
  z_4 \coloneqq z_3 + \frac{2\eps}{L_0}, \quad \text{and} \quad M \coloneqq \Dz + \frac 2 {\Lo} \pare{e^{\frac \Aa 4} - 1}.
\end{align*}
A rough sketch of this construction is plotted in \Cref{fig:lower_bound}. Notably, within the range $\left[z_1, \nicefrac \eta 2\right]$, the gradient increases exponentially with $x$. This steep gradient change is permissible due to the relaxed smoothness assumption. Initiating from $x_1 = 0$ and taking a step, the iterate arrives at $x_2 = \eta$. At this point, we can demonstrate that $F\pare{\ix 2} \geq M$, signifying the emergence of an exponential dependency on $L_1$. Subsequently, along the $x$-axis, the function's value descends with a gradient of $-2\eps$. Iterations will consistently shift to the right due to these negative gradients. Given the algorithm's intrinsic normalization, the shift in $x$ is limited to $\nicefrac \eta t^{\alpha}$ during the $t$-th iteration. To move beyond the interval $\left[z_2, z_3 \right]$ (where gradients remain large at $-2\eps$), the condition $\isum \nicefrac \eta {\ii^{\alpha}} \geq z_3$ must be satisfied, which gives us the lower bound for $T$. This completes the proof sketch of \Cref{lem:momentumlb_functions}.
\end{proof}

It is worth noting that this construction is also applicable to other algorithms and settings, such as \alg{SGD} with diminishing stepsizes under $L$-smoothness. Now we are ready to use \Cref{prop:sufficient_lb} to finish the proof of \Cref{thm:param_agnostic_lb}. 

\begin{proof}[Proof of \Cref{thm:param_agnostic_lb}]
Let $\delta \in (0,1)$ be arbitrary and let $A \in \Ac$ be specified by $\ssi > 0$, $\al \in \pare{\nicefrac 1 2, 1}$ and any momentum generating rule. Define $\eps_0 \coloneqq \nicefrac 1 2, C_1 \coloneqq \max\set{\nicefrac 1 2, \nicefrac 8 \ssi}$ and $K \coloneqq \pare{\nicefrac{\pare{1-\al}}{\ssi}}^{\frac 1 {1-\al}}$. By \Cref{lem:momentumlb_functions} we have that
\begin{align*}
    \cMeasure\pare{A, \Fc_{\Lz, \Lo, \Dz}} \geq K \pare{\frac{\Dz + \frac 1 {\Lo} \pare{e^{\frac \Aa 4} - 1}} \eps}^{\frac 1 {1-\al}}
\end{align*}
for all $\Dz, \Lz, \Lo \geq C_1$. Finally we define $\theta_0 \coloneqq \max \set{C_1, \pare{\nicefrac 4 \ssi}^{\frac 1 \delta}}$ to obtain $\frac{\Aa}{4} \geq \Lo^{1-\delta}$ for all $\Lo \geq \theta_0$. This completes the proof. 
\end{proof}

\section{Experiments}\label{sec:experiments}
In this section, we present experiments designed to empirically validate the theoretical findings of this paper. In concordance with our theory, the primary focus is to demonstrate the robustness of \nsgdm{} to hyperparameter selection in the context of \LzLo-smoothness. Language modeling tasks with LSTM and Transformer architectures are well-known settings for which \LzLo-smoothness was empirically confirmed to be necessary \citep{ComponentWise,Zhang20}. We therefore focus on these tasks.

\begin{figure}
	\centering
	\begin{subfigure}[c]{0.45\textwidth}
		\includegraphics[width=\textwidth]{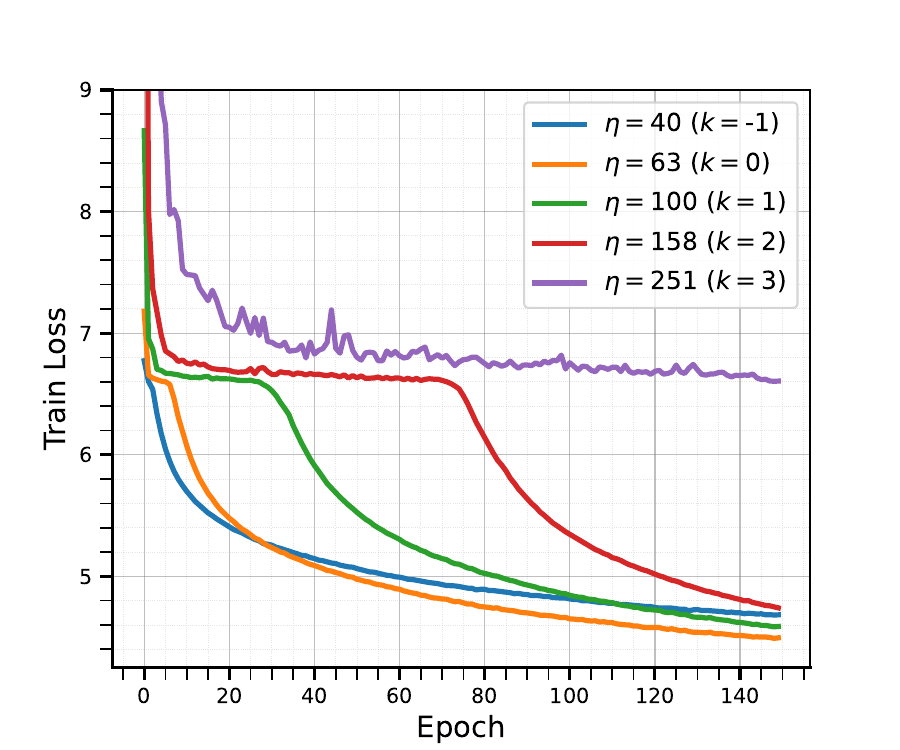}
		\caption{Training curves of \nsgdm.}
		\label{fig:ptb_nsgdm_per_epochs}
	\end{subfigure}
	\hfill 
	\begin{subfigure}[c]{0.45\textwidth}
		\includegraphics[width=\textwidth]{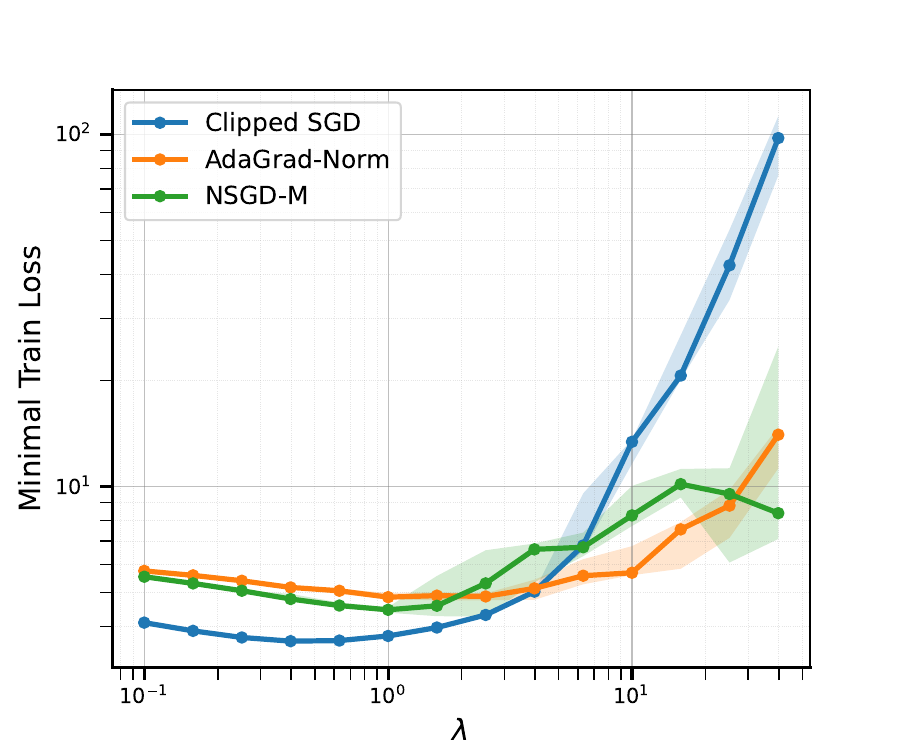}
		\caption{Minimal training loss.}
		\label{fig:ptb_comparison_train}
	\end{subfigure}
	\caption{Results on the PTB dataset. \Cref{fig:ptb_nsgdm_per_epochs} represents the training curves of \nsgdm{} for stepsizes $\ssi = 10^{\nicefrac k 5} \cdot \ssi_{\alg{opt}}$, where $\ssi_{\text{\alg{opt}}} = 63$ and $k \in \set{-1, 0, 1, 2, 3}$. \Cref{fig:ptb_comparison_train} shows the best train loss within $150$ epochs of different algorithms with stepsizes $\lambda \cdot \eta_{\text{\alg{opt}}}$. Shaded areas represent the minimal and maximal value within 5 seeds, the line the median.}
	\label{fig:ptb}
\end{figure}

\paragraph{Experimental Setup.} To match the assumptions of our theory, we conduct training on the Penn Treebank (PTB) \citep{PTB2010Mikolov} and WikiText-2 \citep{Merity2017Wikitext} datasets using the AWD-LSTM architecture \citep{LSTM2018Merity}. Hyperparameters of the model were chosen according to \citep{LSTM2018Merity}. Besides \nsgdm, we also include \alg{AdaGrad-Norm} \citep{BeyondUniform} and Clipped \alg{SGD} \citep{Zhang20}. The clipping threshold for Clipped \alg{SGD} was fixed to be $0.25$ in concordance to previous work \citep{Zhang20}, the decay-rates of \nsgdm{} were chosen according to \Cref{thm:main_result} and $b_0$ of \alg{AdaGrad-Norm} was set to be $b_0 = 10^{-6}$. The code is based on the experiments by \citet{ImprovedAnalysis}.

\paragraph{Penn Treebank.} 
For each algorithm, we select the optimal stepsize $\ssi_{\text{\alg{opt}}}$ using a course grid search in a 50 epoch training. The final training was then carried out for $150$ epochs with stepsizes $\ssi = \lambda \cdot \ssi_{\text{\alg{opt}}}$, where $\lambda = 10^{\nicefrac k 5}, k \in \set{-5, -4, \dots, 8}$. We replicated this procedure with five seeds for reliable results.

\Cref{fig:ptb_nsgdm_per_epochs} shows the behaviour of \nsgdm{} with different stepsizes. The result supports the narrative behind \Cref{thm:main_result} that \nsgdm{} needs an adaption phase before transitioning to a convergence phase. Only after reaching a threshold, \nsgdm{} starts to decrease the loss. 
\Cref{fig:ptb_comparison_train} focuses on the robustness to hyperparameter selection. It compares the smallest training loss across 150 epochs of different algorithms on scaled versions of their optimally tuned stepsize. As expected, well-tuned Clipped \alg{SGD} with constant stepsize outperforms all decaying algorithms, while decaying algorithms are more robust to untuned stepsizes. Between \nsgdm{} and \alg{AdaGrad-Norm} we notice that \nsgdm{} has slightly preferable behaviour for small stepsizes. Furthermore the trend for large stepsizes points towards a more robust behaviour of \nsgdm.

\paragraph{WikiText-2.} For each algorithm we first chose the optimal stepsize $\ssi_{\text{\alg{opt}}}$ based on a course grid search in a 20 epoch training. The final training was then carried out for $150$ epochs with stepsizes $\ssi = \lambda \cdot \ssi_{\text{\alg{opt}}}$, where $\lambda = 10^{\nicefrac k 3}, k \in \set{-3, -4, \dots, 5}$. We replicated this procedure with three seeds for reliable results.

\begin{figure}
	\centering
	\begin{subfigure}[c]{0.45\textwidth}
		\includegraphics[width=\textwidth]{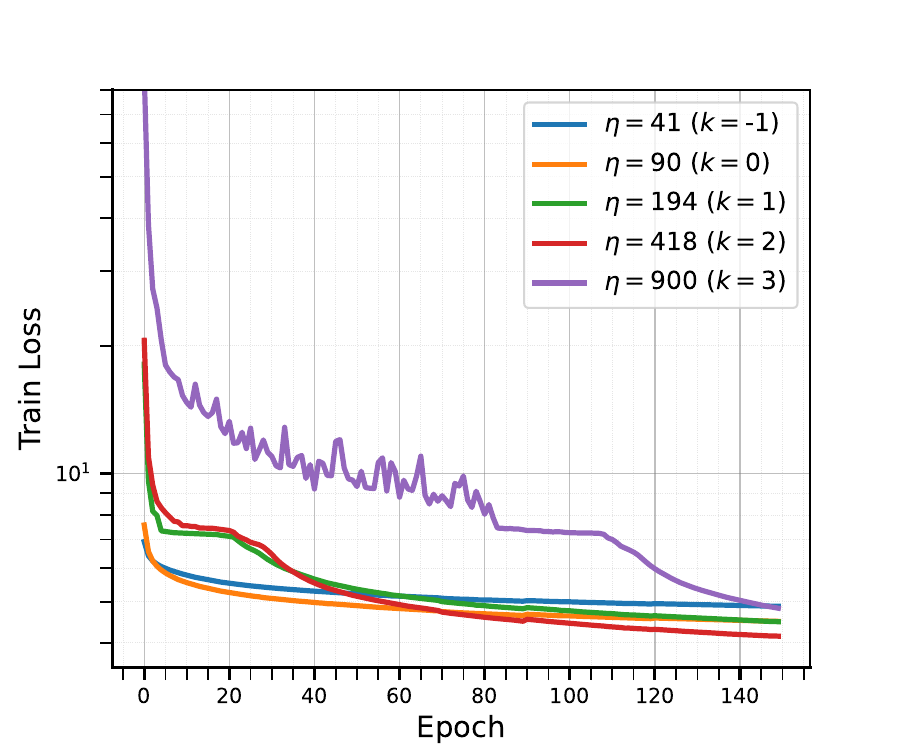}
		\caption{Training curves of \nsgdm.}
		\label{fig:wikitext_nsgdm_per_epoch}
	\end{subfigure}
	\hfill
	\begin{subfigure}[c]{0.45\textwidth}
		\includegraphics[width=\textwidth]{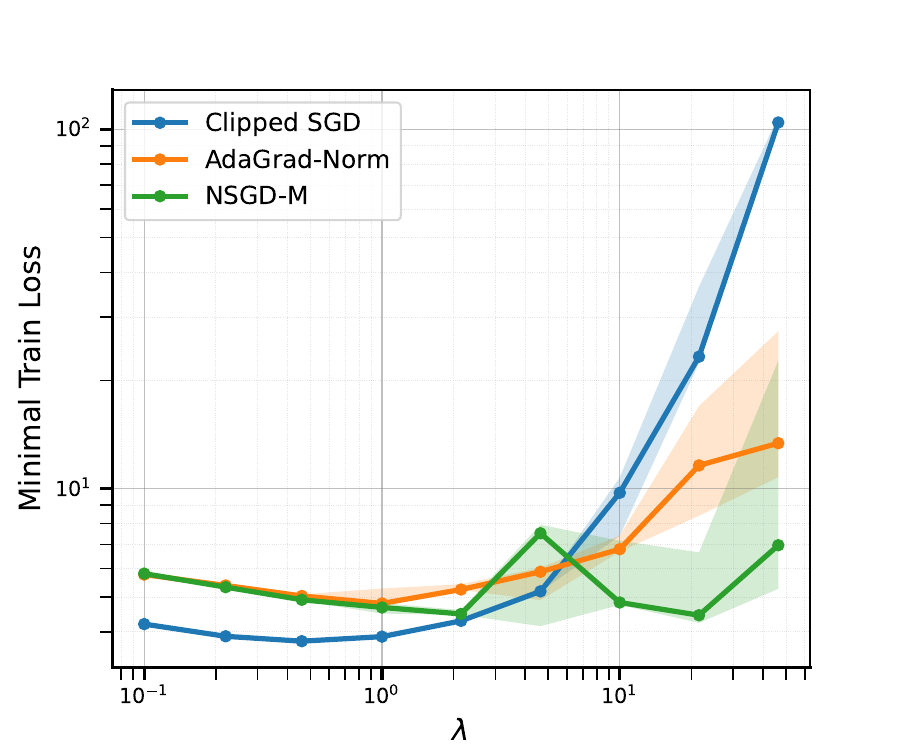}
		\caption{Minimal training loss.}
		\label{fig:wikitext_comparison_train}
	\end{subfigure}
	\caption{Results on the WikiText-2 dataset. \Cref{fig:wikitext_nsgdm_per_epoch} represents the logarithmic training curves of \nsgdm{} for stepsizes $\ssi = 10^{\nicefrac k 3} \cdot \ssi_{\alg{opt}}$, where $\ssi_{\text{\alg{opt}}} = 90$ and $k \in \set{-1, 0, 1, 2, 3}$. \Cref{fig:wikitext_comparison_train} shows the best train loss within $150$ epochs of different algorithms with stepsizes $\lambda \cdot \eta_{\text{\alg{opt}}}$. Shaded areas represent the minimal and maximal value within 3 seeds, the line the median.}
	\label{fig:wikitext}
\end{figure}

In \Cref{fig:wikitext_nsgdm_per_epoch} we can again notice the same threshold behaviours for \nsgdm{} as experienced on the PTB dataset. Instead of a plateau we do however observe higher trainings losses before the fast decrease. Training curves of Clipped \alg{SGD} and \alg{AdaGrad-Norm} can be found in \Cref{fig:wikitext_training_curves}.
\Cref{fig:wikitext_comparison_train} showcases the robustness of \nsgdm{} to hyperparmeter-tuning to an greater extend than \Cref{fig:ptb_comparison_train}. We can see that \nsgdm{} outperforms \alg{AdaGrad-Norm} for nearly all stepsizes, with the gap increasing as stepsizes increase relative to the optimal stepsize. While Clipped \alg{SGD} outperforms the adaptive methods when using the optimally-tuned stepsize or less, it suffers from an order of magnitude higher training loss as stepsizes increase relative to the optimally tuned stepsize. When compared to \Cref{fig:ptb_comparison_train}, a large improvement in performance can be noticed for \nsgdm. We offer the following explanation: While, in both cases, we trained for 150 epochs, the training on the smaller PTB dataset consisted of roughly 680 batches per epoch. On the larger WikiText-2 dataset, epochs consisted of roughly 1500 batches, increasing the total number of iterations from roughly $100 000$ to roughly $230 000$. When assuming similar values of $\Lz, \Lo$, \nsgdm{} hence more likely reached the threshold needed, entering the fast convergence phase, while \alg{AdaGrad-Norm} behaves more steadily, as can be seen in \Cref{fig:wikitext_adagrad_per_epoch}.

\begin{figure}
	\centering
	\begin{subfigure}[c]{0.45\textwidth}
		\includegraphics[width=\textwidth]{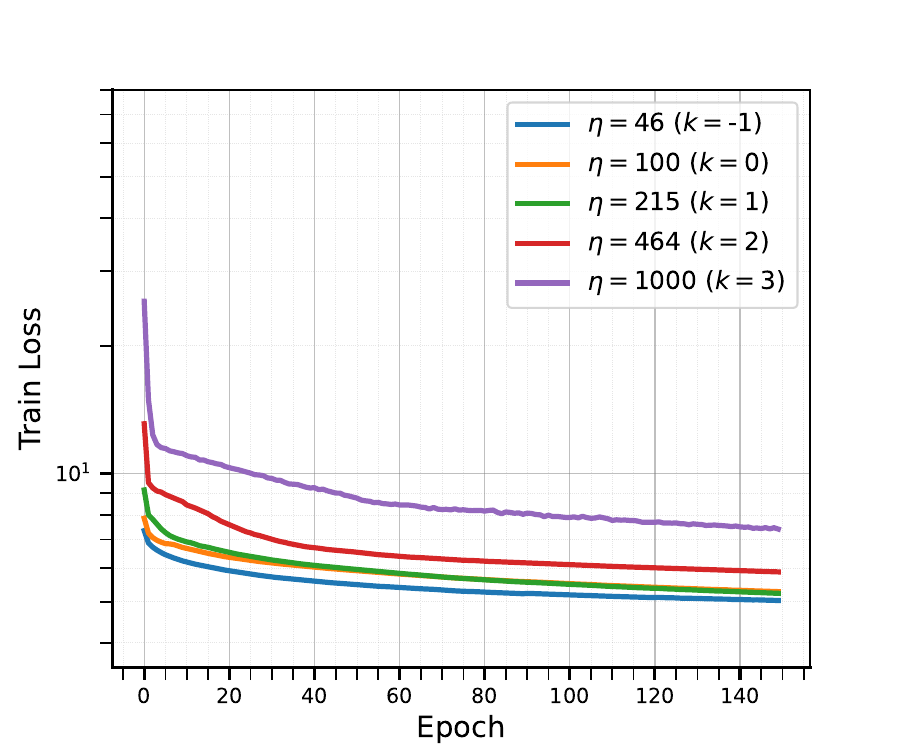}
		\caption{\alg{AdaGrad-Norm} with $\ssi_{\text{\alg{opt}}} = 100$.}
		\label{fig:wikitext_adagrad_per_epoch}
	\end{subfigure}
	\hfill
	\begin{subfigure}[c]{0.45\textwidth}
		\includegraphics[width=\textwidth]{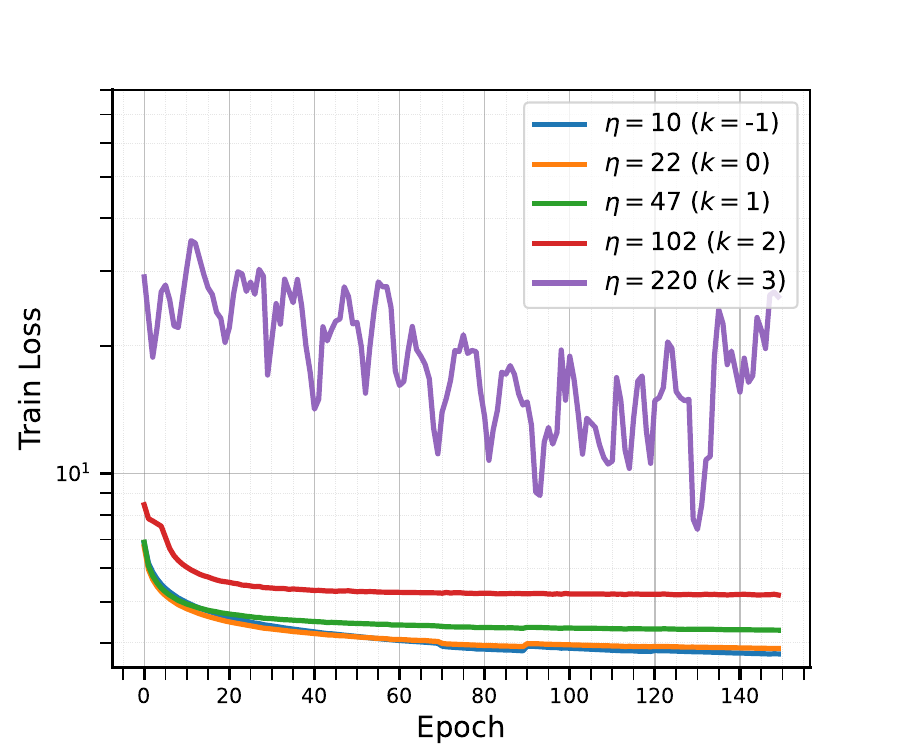}
		\caption{Clipped \alg{SGD} with $\ssi_{\text{\alg{opt}}} = 22$.}
		\label{fig:wikitext_sgd_clip_per_epoch}
	\end{subfigure}
	\caption{Logarithmic training curves of \alg{AdaGrad-Norm} and Clipped \alg{SGD} on WikiText-2 for stepsizes $\ssi = 10^{\nicefrac k 3} \cdot \ssi_{\alg{opt}}$ with $k \in \set{-1, 0, 1, 2, 3}$.}
	\label{fig:wikitext_training_curves}
\end{figure}

\section{Conclusion}\label{sec:concl}
In this work, we conduct a theoretical investigation into parameter-agnostic algorithms under the \LzLo-smoothness assumption. In the stochastic setting, we show that without requiring any knowledge about problem parameters, Normalized Stochastic Gradient Descent with Momentum (\nsgdm) converges at an order-optimal rate, albeit with an exponential term in $L_1$. Further, we introduce a lower bound framework specifically for the parameter-agnostic context, revealing that this exponential term is inescapable for a family of \genmommeth\textcolor{linkcolour}{s}. In the deterministic setting, we show the exponential dependency can be circumvented using \backtrGD~while being parameter-agnostic. 

This work motivates several questions for future research. The most pressing one is whether there exists a fully parameter-agnostic algorithm in the stochastic setting without an exponential term. Another interesting topic is the derivation of lower bounds for all first-order parameter-agnostic methods. 

\bibliography{files/StochasticFirstOrder}

\clearpage
\appendix
\section{Basic Properties of \texorpdfstring{\LzLo}{(L0, L1)}-Smoothness}
\label{sec:app.basic_properties}
In this section, we prove basic properties of \LzLo-Smoothness. We start with the proof of the relation to the original definition by \cite{Zhang20}.

\begin{proof}[Proof of \Cref{lem:equivalence_different_lzlo_defs}]
	``$\Rightarrow$": This implication was already shown by \citet[Corollary A.4]{ImprovedAnalysis}.
	
	``$\Leftarrow$": We slightly adapt the proof by \citet[Proposition 1]{BeyondUniform}. Assume $F$ is \LzLo-smooth according to \Cref{def:lzlo_smoothness}. Let $x,s\in\R^d$ with $\norm s = 1$. For $\al > 0$ our assumption gives
	\begin{align*}
		\norm{\nF \pare{x + \al s} - \nFx} 
		\leq \pare{\Az \pare{\al \Lo} \Lz + \Ao \pare{\al \Lo} \Lo \norm{\nFx}} \al,
	\end{align*}
	and hence,
	\begin{align*}
		\norm{\frac{\nF \pare{x + \al s} - \nFx}{\al}}
		\leq \Az \pare{\al \Lo} \Lz + \Ao \pare{\al \Lo} \Lo \norm{\nFx}.
	\end{align*}
	Using the continuity of norms and the assumption that $F$ is twice continously differentiable, we get
	\begin{align*}
		\Lz + \Lo \norm{\nFx}
		&= \lim_{\al \to 0} \Az \pare{\al \Lo} \Lz + \Ao \pare{\al \Lo} \Lo \norm{\nFx}\\
		&\geq \lim_{\al \to 0} \norm{\frac{\nF \pare{x + \al s} - \nFx}{\al}}\\
		&= \norm{\lim_{\al \to 0} \frac{\nF \pare{x + \al s} - \nFx}{\al}}\\
		&= \norm{\nabla^2F(x) s}.
	\end{align*}
	Taking the $\sup$ over all such $s$ yields the claim.
\end{proof}

The following lemma serves as the \LzLo-smooth counterpart to the well-known quadratic upper bound on the function value change in the $L$-smooth setting.

\begin{lemma}[{c.f.~\cite[Lemma A.3]{ImprovedAnalysis}}]\label{lem:app.lzlo_property}
	Let $d \in \Ngeq$ and $\Lz, \Lo \geq 0$. Assume that $f \colon \R^d \to \R$ is \LzLo-smooth. Then all $x,y\in \R^d$ satisfy
	\begin{align*}
		f(y) \leq 
		f(x) + \nf(x)^\top (y-x) 
		+ \frac 1 2\pare{\Bz \pare{\Lo \norm{x-y}}\Lz 
			+ \Bo \pare{\Lo\norm{x-y}} \Lo \norm{\nf (x)}} \nsq{x-y}, 
	\end{align*}
	where
	\begin{align*}
		\Bz \pare c &= 1 + 2\frac{e^c - 1} c - 4 \frac{e^c - 1 - c}{c^2},\\
		\Bo \pare c &= 2\frac{e^c - 1 - c}{c^2}
	\end{align*}
	tend to 1 as $c$ tends towards 0.
	
	\begin{proof}
		\newcommand{\tempConst}{c}
		This proof closely follows the arguments from \cite{ImprovedAnalysis}. We include the proof for completeness. Let $x,y \in \R^d$ and calculate
		\begin{align*}
			f(y) - f(x) - \nf (x)^\top (y-x)
			&= \int_0^1 \nf \pare{x+t(y-x)}^\top (y-x) dt - \nf (x)^\top (y-x)\\
			&\leq \int_0^1 \norm{\nf \pare{x+t(y-x)} - \nf(x)} \norm{x-y}dt\\
			&\leq \norm{x-y}^2 \pare{\Lz \int_0^1 t\Az \pare{t \tempConst} dt + \Lo \norm{\nf \pare x} \int_0^1 t\Ao \pare{t\tempConst} dt}
		\end{align*}
		where $\tempConst \coloneqq \Lo \norm{x-y}$. We now calculate
		\begin{align*}
			\int_0^1 t\Az \pare{t\tempConst}
			&= \frac 1 2 + \frac{e^c - 1} c - 2 \frac{e^c - 1 - c}{c^2} \eqqcolon \frac 1 2 \Bz \pare{c}
		\end{align*}
		and
		\begin{align*}
			\int_0^1 \Ao(t\tempConst) dt 
			= \frac{e^c-1-c}{c^2} \eqqcolon \frac 1 2 \Bo \pare{c}.
		\end{align*}
		This shows the claim.
	\end{proof}
\end{lemma}

Analogous to the $L$-smooth setting, we can also derive an upper bound for the gradient norm based on the suboptimality gap.

\begin{lemma}[Gradient Bound, {c.f.~\cite[Lemma A.5]{ImprovedAnalysis}}]\label{lem:app.gradient_bound}
	Let $\Lz, \Lo > 0$ and assume that $f\colon \R^d \to \R$ is \LzLo-smooth. Further assume that $f$ is lower bounded by $f^*$. Then all $x \in \R^d$ satisfy
	\begin{align*}
		\min \set{\frac {\norm{\nf (x)}} {\Lo}, \frac {\nsq{\nf(x)}} {\Lz}} \leq 8 \pare{f(x) - f^*}.
	\end{align*}
\end{lemma}

\begin{proof} This proof is again based on \cite{ImprovedAnalysis}. We include it since we require parts of the proof later.
	Let $x \in \R^d$. Firstly note that, for $\Ao$ from \Cref{def:lzlo_smoothness}, the equation
	\begin{align*}
		c = \frac{\Lo \norm{\nf (x)}} {\Ao(c) \Lz + \Lo \Ao (c) \norm{\nf(x)}}
	\end{align*}
	has a solution $c \in (0,1)$. Now we set $\lambda \coloneqq \frac 1 {2\Ao(c)\pare{\Lz + \Lo\norm{\nf(x)}}}$ and	$y \coloneqq x - \lambda \nf (x)$.
	Then \Cref{lem:app.lzlo_property} yields
	\begin{align*}
		f^* \leq f(y)
		&\leq f(x) - \lambda \nsq{\nf(x)} + \Ao\pare{c} \pare{\Lz + \Lo \norm{\nf (x)}} \lambda^2 \nsq{\nf(x)}
		= f(x) - \frac \lambda 2 \nsq{\nf(x)}.
	\end{align*}
	We now differentiate between the two cases $\norm{\nf(x)} \leq \frac {\Lz} {\Lo}$ and $\norm{\nf(x)} > \frac {\Lz} {\Lo}$. Therefore,
	\begin{align*}
		2 \pare{f(x) - f^*} 
		&\geq \frac {\nsq{\nf(x)}} {\Ao(c) \pare{\Lz + \Lo \norm{\nf(x)}}}
		\geq 
		\begin{cases}
			\frac{\nsq{\nf(x)}}{4\Lz}, &\text{if} \norm{\nf(x)} \leq \frac {\Lz} {\Lo}\\
			\frac{\norm{\nf(x)}}{4 \Lo}, &\text{otherwise.}
		\end{cases}
	\end{align*}
	This shows the claim.
\end{proof}

\newpage
\section{Technical Lemmas}
\label{sec:app.technical_lemmas}
This section presents crucial technical lemmas and their proofs. These results may be of interest on their own as they can potentially be applied in the analysis of other momentum-based algorithms.

{\renewcommand{\li}{b}
	\begin{lemma}[Technical Lemma]
		\label{lem:technical.general}
		Let $q \in (0,1), p \geq 0$ and $\ii > 0$. Further let $a,b \in \N_{\geq 2}$ with $a \leq b$. Then the following statements are true.
		\begin{itemize}
			\item[$i$)] We have
			\begin{align*}
				\prod_{\ii = a}^b \pare{1-\ii^{-q}}
				&\leq \exp\pare{\frac 1 {1-q} \pare{a^{1-q} - b^{1-q}}}.
			\end{align*}
			\item[$ii$)] If $p \geq q$, then
			\begin{align*}
				\sum_{\ii = a}^\li t^{-p} \prod_{\tau = a}^\ii \pare{1-\tau^{-q}}
				&\leq \frac 
				{\pare{a-1}^{q-p} \exp\pare{\frac{a^{1-q} - \pare{a-1}^{1-q}}{1-q}} - \li^{q-p} \exp\pare{\frac{a^{1-q} - \li^{1-q}}{1-q}}}
				{1 + \pare{p-q}\li^{q - 1}},
			\end{align*}
			and in particular,
			\begin{align*}
				\sum_{\ii = a}^\li t^{-p} \prod_{\tau = a}^\ii \pare{1-\tau^{-q}}
				&\leq \pare{a-1}^{q-p} \exp\pare{\frac{a^{1-q} - \pare{a-1}^{1-q}}{1-q}} = \Oc \pare{a^{q-p}}.
			\end{align*}
			\item[$iii$)] (c.f.~{\cite[Lemma 15]{Ilyas23}}\footnote{Note that the proof in the paper has a typo in the last line of page 42. Instead of $(1-q)$ the authors meant $(1-q)^{-1}$.}) 
			If $a \geq p^{\frac 1 {1 - q}}$ and $a \geq \pare{\frac{p - q} 2}^{\frac 1 {1-q}}$, then
			\begin{align*}
				\sum_{\ii = a}^b \ii^{-p} \prod_{\tau = \ii + 1}^b \pare{1-\tau^{-q}}
				\leq 2\exp\pare{\frac 1 {1-q}}\pare{b+1}^{q-p}.
			\end{align*}
			Note that these requirements are always fulfilled for $p \leq 1$.
		\end{itemize}
	\end{lemma}
	\begin{proof}
		\newcommand{\omq}{{1-q}}
		\newcommand{\qmp}{{q-p}}
		\newcommand{\g}[1]{\exp \pare{-\frac {{#1}^\omq} \omq}}
		\newcommand{\gtt}{\g \ii}
		\newcommand{\gtwo}[1]{\exp \pare{\frac {{#1}^\omq} \omq}}
		\newcommand{\gtwot}{\gtwo \ii}
		$i)$ The first claim follows from the calculation
		\begin{align}\label{eq:technical.upper_bound_prod_beta}
			\begin{split}
				\prod_{\ii = a}^b \pare{1-t^\mq}
				&\leq \exp \pare{- \sum_{\tau = a}^b \ii^\mq} \leq \exp \pare{- \int_a^{b+1} t^\mq dt} 
				= \exp\pare{\frac 1 \omq \pare{a^{1-q} - \pare{b+1}^\omq}},
			\end{split}
		\end{align}
		where we used $1-x \leq e^{-x}$ in the first, and the monotonicity of $t^\mq$ in the second inequality. Weakening the inequality by replacing $\pare{b+1}$ with $b$ finishes the proof.\\
		$ii$) For the second inequality we use $i$) to derive
		\begin{align*}
			\sum_{\ii = a}^\li \ii^\mpp \prod_{\tau = a}^\ii \pare{1-\ii^\mq}
			&\leq \exp\pare{\frac {a^\omq} \omq} \sum_{\ii = a}^\li \ii^\mpp \exp \pare{- \frac{\ii^\omq}\omq}.
		\end{align*}
		Using the monotonicity of $\ii^\mpp \exp \pare{-\ii^\omq}$ we obtain
		\begin{align*}
			\sum_{\ii = a}^\li \ii^\mpp \exp \pare{- \frac{\ii^\omq}\omq}
			&\leq \int_{a-1}^\li \ii^\mpp \exp\pare{- \frac{\ii^\omq} \omq} d\ii
			= \int_{a-1}^\li \ii^\qmp \ii^\mq \exp \pare{- \frac{\ii^\omq} \omq} d\ii.
		\end{align*}
		Partial integration now yields
		\begin{align*}
			&\int_{a-1}^\li \ii^\qmp \ii^\mq \gtt d\ii\\
			=&\ \left[ - \ii^\qmp \gtt \right]_{\ii = a-1}^{\ii = \li}
			- \pare{p-q}\int_{a-1}^\li \ii^{\qmp - 1}\gtt dt\\
			=&\ \pare{a-1}^\qmp \g {\pare{a-1}} - \li^\qmp \g \li
			 + \pare{\qmp} \int_{a-1}^\li \ii^{\qmp - 1}\gtt dt.
		\end{align*}
		Finally, we use that $\ii^{\qmp - 1}\gtt$ is monotonically decreasing and $p \geq q$ to derive
		\begin{align*}
			\pare{\qmp} \int_{a-1}^\li \ii^{\qmp - 1}\gtt dt
			&\leq \pare \qmp \li^{q-1} \int_{a-1}^\li \ii^\mpp \gtt dt.
		\end{align*}
		Noting that this is the integral we started with and rearranging yields the claim.\\
		$iii$) The proof of the last claim uses the same arguments as in \citep{Ilyas23}. First we use $i$) to obtain
		\begin{align*}
			\sum_{\ii = a}^b \ii^{-p} \prod_{\tau = \ii + 1}^b \pare{1-\tau^{-q}}
			&\leq \sum_{\ii = a}^b \ii^{-p} \exp\pare{- \sum_{\tau = \ii + 1}^b \tau ^\mq}
			= \exp \pare{-\sum_{\tau = 1}^b \tau^\mq} \sum_{\ii = a}^b \ii^{-p} \exp\pare{\sum_{\tau = 1}^\ii \tau ^\mq}.
		\end{align*}
		Using the monotonicity of $\tau^\mq$, we get
		\begin{align*}
			\exp \pare{-\sum_{\tau = 1}^b \tau^\mq} 
			\leq \exp \pare{- \int_1^{b+1} \tau^\mq d\tau}
			= \exp\pare{\frac{1 - \pare{b+1}^\omq} \omq}
		\end{align*}
		and
		\begin{align*}
			\exp\pare{\sum_{\tau = 1}^\ii \tau ^\mq}
			\leq \exp \pare{\int_0^\ii \tau^\mq d\tau}
			= \exp\pare{\frac{\ii^\omq} \omq}.
		\end{align*}
		We now proceed to bound 
		\begin{align*}
			\sum_{\ii = a}^b \ii^\mpp \exp\pare{\sum_{\tau = 1}^\ii \tau ^\mq}
			\leq \sum_{\ii = a}^b \ii^\mpp \exp\pare{\frac{\ii^\omq} \omq}.
		\end{align*}
		Therefore, note that $f\pare\ii \coloneqq \ii^\mpp \exp\pare{\frac{\ii^\omq} \omq}$ is monotonically increasing for $\ii \geq a$ by our assumption on $a$. This implies
		\begin{align*}
			\sum_{\ii = a}^b \ii^\mpp \exp\pare{\frac{\ii^\omq} \omq}
			&\leq \int_a^{b+1} \ii^\mpp \exp\pare{\frac{\ii^\omq} \omq} dt \eqqcolon I.
		\end{align*}
		Integration by party now yields
		\begin{align*}
			I
			&= \int_a^{b+1} \ii^\qmp  \ii^\mq \exp\pare{\frac{\ii^\omq} \omq} dt\\
			&= \left[ \ii^\qmp \gtwot \right]_{\ii = a}^{\ii = b+1}
			- \pare{\qmp}\int_a^{b+1} \ii^{\qmp - 1} \gtwot dt\\
			&\leq \pare{b+1}^\qmp \gtwo{\pare{b+1}} - a^\qmp \gtwo a + \pare{p-q}a^{q-1}I,
		\end{align*}
		where we used $p \geq q$ in the last inequality. By our second assumption on $a$ we now get that $\pare{p-q}a^{q-1} \leq \nicefrac 1 2$ and hence
		\begin{align*}
			I \leq 2\pare{b+1}^\qmp \gtwo{\pare{b+1}} - 2a^\qmp \gtwo a.
		\end{align*}
		Putting together the pieces yields
		\begin{align*}
			\sum_{\ii = a}^b \ii^{-p} \prod_{\tau = \ii + 1}^b \pare{1-\tau^{-q}}
			&\leq 2\exp\pare{\frac{1 - \pare{b+1}^\omq} \omq} \pare{\pare{b+1}^\qmp \gtwo{\pare{b+1}} - a^\qmp \gtwo a}\\
			&= 2\exp\pare{\frac 1 \omq}\pare{b+1}^\qmp - a^\qmp \exp\pare{\frac{1 - \pare{b+1}^\omq +a^\omq} \omq},
		\end{align*}
		thus proving the last claim.
	\end{proof}
}

The following lemma applies the specific values of $p$ and $q$ to \Cref{lem:technical.general}.

\begin{lemma}[Technical Lemma]\label{lem:technical.explicit}
	Let $\eta > 0$ and for $\ii \in \Ngeq$ we set
	\begin{align*}
		\betat &\coloneqq 1 - \ii^{-\nicefrac 1 2},\\
		\etat &\coloneqq \eta \ii^{-\nicefrac 3 4}.
	\end{align*}
	Then we have
	\begin{enumerate}[label*=\alph*)]
		\item For all $\li \in \Ngeq$ the following inequalities hold: 
		\begin{enumerate}
			\item[$i$)] 
				$\isum \etat \prod_{\tau = 2}^t \iterationbeta \tau \leq \frac 7 2 \eta$;
			\item[$ii$)] 
				$\isum \etat \sqrt{\sum_{\tau = 1}^\ii  \iterationa \tau ^2 \prod_{\kappa = \tau + 1}^t \iterationbeta \kappa^2}
				\leq \eta \pare{ \frac 7 2 + \sqrt{2e^2} \log \pare \li}.$
		\end{enumerate}
		\item Let $\li \in \Ngeq$ and define $
			\Ct \coloneqq 1 + 2 \frac{e^{\Lo \etat} - 1}{\Lo \etat} - 4 \frac{e^{\Lo \etat} - 1 - \Lo \etat}{\pare{\Lo \etat}^2}$.
		Then the following inequalities hold: 
		\begin{enumerate}
			\item[$i$)]
				$\isum \etat^2 \Ct 
				\leq 6\eta^2\frac {e^{\Lo \eta} - 1}{\Lo \eta}$;
			\item[$ii$)]
				$\isum \etat \sum_{\tau = 2}^\ii \iterationeta \tau \iterationC \tau \prod_{\kappa=\tau}^\ii \iterationbeta \kappa
				\leq 7 \ssi^2 \pare{3 \frac{e^\Aa - 1} \Aa + \log \pare \li}$.
		\end{enumerate}
		\newcommand{\sz}{{s_0}}
		\item For $t \in \Ngeq$ we define 
			$\Dt \coloneqq 2 \frac{e^{\Lo \etat} - 1 - \Lo \etat}{\pare{\Lo \etat}^2}$ and
			$\dt \coloneqq 4\ssi \pare{\ii - 1}^\of - 3 \ssi.$
		Then for all $b \in \N_{\geq 2}$, the following inequalities hold: 
		\begin{enumerate}
			\item[$i$)]
				$\sum_{\ii = 2}^b \Lo \etat \Dt \ii^{-\of} \dt e^{\Lo \dt}
				\leq \frac 1 2 \ssi^2 \Lo e^{2\Aa} + 4 \ssi e^{-\frac 5 2 \Aa} \pare{e^{4 \Aa b^\of} - e^{4 \Aa}}$;
			\item[$ii$)]
				$\sum_{\ii = 1}^b \Lo \etat \Dt \ii^{-\of} e^{\Lo \dt}
				\leq \frac 3 2 \Aa e^{\frac 5 3 \Aa} + e^{-\frac 5 2\Aa} \pare{e^{4 \Aa b^\of} - e^{4\Aa}}$;
			\item[$iii$)] If additionally $\Aa \geq \frac 1 2$, we have
				$\sum_{\ii = 1}^b \Lo \etat \Dt \ii^{-\of} e^{\Lo \dt}
				\leq \frac 3 2 \Aa e^{\frac 5 3 \Aa} + e^{-\frac 5 2\Aa}\pare{2 b^{-\of} e^{4\Aa b^\of} - e^{4 \Aa}}.$
		\end{enumerate}
	\end{enumerate}
\end{lemma}

\begin{proof}
	Let $\li \in \Ngeq$ and denote $p \coloneqq \nicefrac 3 4, q \coloneqq \nicefrac 1 2$ for simplicity.
	\begin{itemize}
		\item[$a$) $i$)]
		The inequality follows from
		\begin{align*}
			\isum \etat \prod_{\tau = 2}^t \iterationbeta \tau
			= \eta + \sum_{\ii = 2}^\li \etat \prod_{\tau = 2}^\ii \iterationbeta \tau
			\leq \eta + \eta\exp\pare{2\sqrt 2 - 2}
			\leq \frac 7 2 \eta,
		\end{align*}
		where we used \Cref{lem:technical.general} $ii$) in the first inequality.
		
		\item[$a$) $ii$)]
		We start by regrouping
		\begin{align*}
			\isum \etat \sqrt{\sum_{\tau = 1}^\ii  \iterationa \tau ^2 \prod_{\kappa = \tau + 1}^t \iterationbeta \kappa^2}
			&<  \isum \etat \pare{\prod_{\kappa=2}^\ii \pare{1-\kappa^{-q}} + \sqrt{\sum_{\tau = 2}^\ii \tau^{-2q} \prod_{\kappa = \tau + 1}^t \pare{1- \kappa^{-q}}}}.
		\end{align*}
		Applying \Cref{lem:technical.general} $i$), $iii)$ and $a$) $i$) now yields the statement:
		\begin{align*}
			\isum \etat \sqrt{\sum_{\tau = 1}^\ii  \iterationa \tau ^2 \prod_{\kappa = \tau + 1}^t \iterationbeta \kappa^2}\ \ \
			\stack{i), \ref{lem:technical.general}}{\leq} \frac 7 2 \eta + \sum_{\ii = 2}^\li \etat \sqrt{2e^2 \pare{\ii + 1}^{-q}}
			\leq \eta \pare{ \frac 7 2 + \sqrt{2e^2} \log \pare \li}.
		\end{align*}
		Note that the first inequality is rather loose, a more precise analysis might yield a better result. The above result does however suffice for our use-case.
		
		\item[$b$) $i$)] 
		First note that
		\begin{align}\label{lzlo.infnsgd.eq:explicit_technical_bounds_Ct_bound}
			\Ct \leq 2 \frac{e^\Aat - 1} \Aat
		\end{align}
		and hence
		\begin{align}\label{lzlo.infnsgd.eq:explicit_technical_bounds_bi_1}
			\isum \etat^2 \Ct
			\leq \frac 2 {\Lo} \isum \etat \pare{e^\Aat - 1}.
		\end{align}
		Now we calculate
		\begin{align}\label{lzlo.infnsgd.eq:explicit_technical_bounds_bi_2}
			\begin{split}
				\isum \etat \pare{e^\Aat - 1}
				&\leq \eta^2 \frac{e^\Aa - 1} \eta + \int_{1}^\li \etat \pare{e^\Aat - 1} dt
				= \eta^2 \pare{\frac{e^\Aa - 1} \eta + \frac 1 \eta \int_1^\li \ii^{-p} \pare{e^\Aat - 1}}
			\end{split}
		\end{align}
		and further
		\begin{align}\label{lzlo.infnsgd.eq:explicit_technical_bounds_bi_3}
			\int_1^\li \ii^{-p} \pare{e^\Aat - 1}
			&= \int_1^\li t^{-p} \sum_{k=1}^\infty \frac{\pare{\Aa t^{-p}}^k}{k!}dt
			= \sum_{k=1}^\infty \int_1^\li \frac{\pare{\Aa}^k}{k!} t^{-p(k+1)}dt,
		\end{align}
		where we used that the exponential series converges locally uniformly in the second equality. Finally we calculate for $k \geq 2$
		\begin{align}\label{lzlo.infnsgd.eq:explicit_technical_bounds_bi_4}
			\int_1^\li \ii^{-p \pare{k+1}} dt
			&= \frac 4 {-3k + 1}\pare{\li^{-p\pare{k+1} + 1} - 1} \leq \frac 4 {3k - 1} \leq \frac 4 {k + 1}.
		\end{align}
		Combining \eqref{lzlo.infnsgd.eq:explicit_technical_bounds_bi_3} and \eqref{lzlo.infnsgd.eq:explicit_technical_bounds_bi_4} now yields
		\begin{align}\label{lzlo.infnsgd.eq:explicit_technical_bounds_bi_5}
			\int_1^\li \ii^{-p} \pare{e^\Aat - 1}
			&\leq 4 \sum_{k=1}^\infty \frac{\pare{\Aa}^k}{\pare{k+1}!}
			= 4 \frac{e^\Aa - 1 - \Aa}{\Aa}
		\end{align}
		and hence
		\begin{align*}
			\isum \etat^2 \Ct
			\stackAlign{\eqref{lzlo.infnsgd.eq:explicit_technical_bounds_bi_1}}{\leq}
			\frac 2 {\Lo} \isum \etat \pare{e^\Aat - 1}\\
			\stackAlign{\eqref{lzlo.infnsgd.eq:explicit_technical_bounds_bi_2}}{\leq}
			\eta^2 \pare{2 \frac{e^\Aa - 1}{\Aa} + \frac 2 \Aa \int_1^\li \ii^{-p} \pare{e^\Aat - 1} dt}\\
			\stackAlign{\eqref{lzlo.infnsgd.eq:explicit_technical_bounds_bi_5}}{\leq}
			\eta^2 \pare{2 \frac{e^\Aa - 1}{\Aa} + 8 \frac{e^\Aa - 1 - \Aa}{\pare \Aa ^2}}.
		\end{align*}
		The claim now follows by noting that for all $x \geq 0$ the inequality $2\frac{e^x - 1 - x}{x^2} \leq \frac{e^x - 1}{x}$ is satisfied.
		
		\item[$b$) $ii$)] Firstly, $a$) $i$) yields
		\begin{align}\label{lzlo.infnsgd.eq:explicit_technical_bounds_bii_1}
			\isum \sum_{\tau = 2}^\ii \etat \iterationeta \tau \pare{\prod_{\kappa = \tau}^\ii \iterationbeta \kappa} \iterationC \tau
			&= \sum_{\tau = 2}^\li \iterationeta \tau \iterationC \tau \sum_{\ii = \tau}^\li \etat \prod_{\kappa=\tau}^\ii \iterationbeta \kappa
			\stackrel{a) i)}{\leq} \frac 7 2 \eta \sum_{\tau = 2}^\li \iterationeta \tau \iterationC \tau \pare{\tau - 1}^{q - p}
			\leq \frac 7 2 \eta^2 \sum_{\tau = 1}^\li \tau^{-1} \iterationC \tau.
		\end{align}
		To upper bound $\isum \ii^{-1} \Ct$ we first use \eqref{lzlo.infnsgd.eq:explicit_technical_bounds_Ct_bound} to get
		\begin{align*}
			\isum \ii^{-1} \Ct
			\stackrel{\eqref{lzlo.infnsgd.eq:explicit_technical_bounds_Ct_bound}}{\leq}
			2 \isum \ii^{-1} \frac{e^\Aat - 1}{\Aat}
			\leq 2 \frac{e^\Aa - 1} \Aa + 2\sum_{\ii = 2}^\li \ii^{-1} \frac{e^\Aat - 1}{\Aat}.
		\end{align*}
		Next we focus on bounding $\sum_{\ii = 2}^\li \ii^{-1} \frac{e^\Aat - 1}{\Aat}$. We therefore again use the locally uniform convergence of the exponential series to get
		\begin{align*}
			\sum_{\ii = 2}^\li \ii^{-1} \frac{e^\Aat - 1}{\Aat}
			&\leq \int_1^\li \ii^{-1} \sum_{k=1}^\infty \frac{\pare{\Aat}^{k-1}}{k!}dt\\
			&= \frac 1 {\Aa} \sum_{k=1}^\infty \frac{\pare \Aa^k}{k!} \int_1^\li \ii^{-pk - \nicefrac 1 4}dt\\
			&= \log \pare \li + \frac 1 {\pare \Aa^2} \sum_{k=2}^\infty \frac {\pare \Aa^{k+1}} {k!} \frac{1 - \li^{-p\pare{k-1}}}{p \pare{k-1}}\\
			&\leq \log \pare \li + \frac 4 {\pare \Aa^2} \sum_{k=2}^\infty \frac{\pare \Aa^{k+1}}{ \pare{k+1}!},
		\end{align*}
		where we used that $3(k-1) \geq k + 1$ for $k \geq 2$. Putting the above together and using that $2\frac{e^x - 1 - x - x^2/2}{x^2} \leq \frac{e^x - 1}{x}$ for $x \geq 0$ gives
		\begin{align*}
			\isum \sum_{\tau = 2}^\ii \etat \iterationeta \tau \pare{\prod_{\kappa = \tau}^\ii \iterationbeta \kappa} \iterationC \tau
			\leq \frac 7 2 \eta^2 \pare{2 \frac{e^\Aa - 1} \Aa + 2 \log \pare \li + 4 \frac{e^\Aa - 1} \Aa}
			= 7 \ssi^2 \pare{\log \pare \li + 3 \frac{e^\Aa - 1} \Aa}
		\end{align*}
		and hence proves the claim.
		\item[$c$) $i$)] We start off by calculating
		\begin{align*}
			\wsum 2 \Lo \etat \Dt \ii^{-\of} \dt e^{\Lo \dt}
			&\leq \Lo \iterationeta 2 \iterationD 2 2^{-\of}\ssi e^\Aa
			+ \wsum 3 \Lo \etat \Dt \ii^{-\of} \dt e^{\Lo \dt}\\
			&\leq \frac 1 2 \ssi^2 \Lo e^{2 \Aa}
			+ 4 \ssi \wsum 3 \Lo \etat \Dt e^{\Lo \dt}
		\end{align*}
		and further \newcommand{\wint}{\int_2^{b+1}}
		\begin{align}\label{lzlo.infnsgd.lem:explicit_technical_bounds_ci_1}
			\begin{split}
				\wsum 3 \Lo \etat \Dt e^{\Lo \dt}
				&\leq \wsum 3 \Lo \etat \exp \pare{\Lo \pare{4\ssi \pare{\ii - 1}^\of - 3 \ssi + \etat}}\\
				&\leq e^{-\frac 5 2\Aa} \wsum 3 \Lo \iterationeta {\ii - 1} e^{4 \Aa \pare{\ii-1}^\of}\\
				&\leq e^{-\frac 5 2\Aa} \wint \Aa \pare{\ii - 1}^{-p} e^{4 \Aa \pare{\ii - 1}^{1-p}} dt\\
				&= e^{-\frac 5 2\Aa} \pare{e^{4 \Aa b^\of} - e^{4\Aa}}.
			\end{split}
		\end{align}
		Here we used that $g \pare \ii \coloneqq \Lo \etatm e^{4 \ssi \Lo \pare{\ii - 1}^\of}$ is non-negative and monotonically decreasing before turning monotonically increasing in the third inequality. Noting that \eqref{lzlo.infnsgd.lem:explicit_technical_bounds_ci_1} also holds for $b = 2$ yields the claim.
		\item[\phantom{$c$)} $ii$)] We have
		\begin{align}\label{lzlo.infnsgd.eq:explicit_technical_bounds_cii_1}
			\begin{split}
				\sum_{\ii = 1}^b \Lo \etat \Dt \ii^{-\of} e^{\Lo \dt}
				&= \Aa\iterationD \fin + \frac 1 2 \Aa \iterationD 2 e^\Aa + \wsum 3 \Lo \etat \Dt \ii^{-\of} e^{\Lo \dt}\\
				&\leq \Aa e^\Aa + \frac 1 2 \Aa e^{\pare{1+2^{-\frac 3 4}}\Aa} + \wsum 3 \Aat \Dt e^{\Lo \dt}
			\end{split}
		\end{align}
		and using \eqref{lzlo.infnsgd.lem:explicit_technical_bounds_ci_1} yields
		\begin{align*}
			\sum_{\ii = 1}^b \Lo \etat \Dt \ii^{-\of} e^{\Lo \dt}
			&\leq \frac 3 2 \Aa e^{\frac 5 3 \Aa} + e^{-\frac 5 2\Aa} \pare{e^{4 \Aa b^\of} - e^{4\Aa}}.
		\end{align*}
		\item[\phantom{$c$)} $iii$)] We first again calculate
		\begin{align*}
			\wsum 3 \Lo \etat \Dt \ii^{-\of} e^{\Lo \dt}
			&\leq e^{-\frac 5 2 \Aa} \wint \ii^{-\of} \Lo \ssi \pare{\ii-1}^{-\frac 3 4} e^{4\Aa \pare{\ii - 1}^\of}dt
		\end{align*}
		before, similar to the proof of \Cref{lem:technical.general} $iii$), using partial integration to derive
		\begin{align*}
			I 
			&\coloneqq \wint \ii^{-\of} \Lo \ssi \pare{\ii-1}^{-\frac 3 4} e^{4\Aa \pare{\ii - 1}^\of}dt\\
			&= \left[ \ii^{-\of} e^{4\Aa \pare{b-1}^\of} \right]_{\ii = 2}^{\ii = b + 1}
			+ \of \wint \ii^{- \frac 5 4} e^{4\Aa\pare{\ii - 1}^\of}dt\\
			&\leq b^{-\of} e^{4\Aa b^\of} - \frac 1 2 e^{4 \Aa} + \frac 1 {2^\of 4 \Aa} \wint \Aa \pare{\ii-1}^{-1} e^{4 \Aa \pare{\ii-1}^\of}\\
			&\leq b^{-\of} e^{4\Aa b^\of} - \frac 1 2 e^{4 \Aa} + \frac 1 {4 \Aa} I.
		\end{align*}
		By our assumption we have $\frac 1 {4 \Aa} \leq \frac 1 2 $ and hence
		\begin{align*}
			I \leq 2 b^{-\of} e^{4\Aa b^\of} - e^{4 \Aa}.
		\end{align*}
		Finally \eqref{lzlo.infnsgd.eq:explicit_technical_bounds_cii_1} yields the claim.
	\end{itemize}
	
\end{proof}

\newpage
\section{Missing Proofs}
\label{sec:app.missing_proofs}
This section contains the proofs for \Cref{sec:upper_bounds} and \Cref{sec:lower_bounds}.

\subsection{Proofs for Parameter-Agnostic Upper Bounds}\label{sec:app.missing_proofs.upper_bounds}
\subsubsection{Stochastic Setting}
\label{sec:app.missing_proofs.upper_bounds.stochastic_setting}

We start with the proof of \Cref{thm:main_result}, which has the same structure as in the $L$-smooth setting \citep{nsgdm}: We first derive a Descent Lemma, second bound the momentum deviation $\norm{\mt - \nFxk}$ and third combine these two to show the result. The last step is however more intricate, as large stepsizes in the beginning can lead to an exponential increase in the gradient norm. The main intuitions behind the third step are the following:

Due to potentially too large stepsizes, we cannot use the descent lemma to control the expected gradient norm in the beginning. Only after reaching a threshold $t_0 \propto \pare \Aa^4$ the gradient norms can be controlled in this fashion. Before this threshold, in the \emph{adaption phase}, we instead use \LzLo-smoothness to control the gradient norms based on $\norm{\nF \pare{\xfin}}$.
After this threshold, in the \emph{convergence phase}, \Cref{lem:technical.explicit} essentially establishes that the diminishing step-size rule $\etat = \ii^{-p}$ exhibits the same asymptotically behaviour as if the stepsizes were chosen constantly as $\etat \equiv \li^{-p}$, where $\li$ denotes the iteration horizon. This aligns with the behaviour of \nsgdm{} in the $L$-smooth setting \citep{Junchi2022}. In particular, this implies that $p = \nicefrac 3 4$ is the only possible choice to achieve the optimal complexity \citep{nsgdm, ImprovedAnalysis}.

Unless stated otherwise, the notations $\set{\iterationXi{1}, \iterationXi{2}, \dots}, \set{\iterationG{1}, \iterationG{2}, \dots}, \set{\iterationm{1}, \iterationm{2}, \dots}$ and $\set{\iterationX{1}, \iterationX{2}, \dots}$ correspond to the iterations generated by \nsgdm~throughout this section. We denote the natural filtration of $\ls{\iterationxi \fin}{\xit}$ with respect to the underlying probability space by $\Fct \coloneqq \sigma\pare{\ls{\iterationxi \fin, \iterationxi 2}{\xit}}$.

\begin{lemma}[Descent Lemma]\label{lem:analysis.descent_lemma}
	Assume \assumlzlo~and let $t \in \N_{\geq 2}$. Then
	\begin{align*}
		F \atxkp - F \atxk 
		\leq - \etat \norm{\nFxk} + 2\etat \norm{\nFxk - \mt} + \frac {\etat^2} 2 \pare{\Lz \Ct + \Lo \Dt \norm \nFxk},
	\end{align*}
	where $\Ct \coloneqq \Bz \pare{\Lo \etat}$ and $\Dt \coloneqq \Bo \pare{\Lo \etat}$, where $\Bz, \Bo$ are as defined in \Cref{lem:app.lzlo_property}.
	If we further assume \assumlb~we also get
	\begin{align*}
		\isum \pare{\etat - \frac{\Lo \etat^2 \Dt} 2} \norm \nFxk
		&\leq \Dz + \frac \Lz 2 \isum \etat^2 \Ct  + 2 \isum \etat \norm{\nFxk - \mt},
	\end{align*}
	where $\Dz \coloneqq F \pare{\xfin} - F^*$.
\end{lemma}

\begin{proof}
    The proof follows the arguments by \citet{LzLoNGD}. Using \Cref{lem:app.lzlo_property} we get
	\begin{align*}
		F \atxkp - F \atxk 
		&\leq \nFxk^\top \pare{\xkp - \xk} + \frac {\etat^2} 2 \pare{\Lz \Ct + \Lo \Dt \norm \nFxk}\\
		&= - \frac \etat {\norm \mt} \nF(x)^\top \mt + \frac {\etat^2} 2 \pare{\Lz \Ct + \Lo \Dt \norm \nFxk}\\
		&= - \frac \etat {\norm \mt} \pare{\nFxk - \mt}^\top \mt - \etat \norm \mt + \frac {\etat^2} 2 \pare{\Lz \Ct + \Lo \Dt \norm \nFxk}.
	\end{align*}
	Utilizing Cauchy-Schwarz and $\etat \norm \nFxk \leq \etat \norm{\nFxk - \mt} + \etat \norm \mt$ now yields
	\begin{align*}
		F \atxkp - F \atxk 
		&\leq - \etat \norm \nFxk + 2\norm{\nFxk - \mt} + \frac {\etat^2} 2 \pare{\Lz \Ct + \Lo \Dt \norm \nFxk}
	\end{align*}
	and hence the first claim. For the second statement we sum up to get
	\begin{align*}
		\isum \pare{\etat - \frac{\Lo \etat^2 \Dt} 2} \norm{\nF \pare x}
		\leq \Dz + \frac 1 2 \isum \Lz \etat^2 \Ct + 2 \isum \etat \norm{\nFxk - \mt}.
	\end{align*}
\end{proof}

\begin{lemma}[General Momentum Deviation Bound]\label{lem:analysis.mu_bound}
	Assume \assumlzlo, \assumbounded~and let $\ii \in \N_{\geq 1}$. Suppose $\iterationbeta \fin = 0$. Then we have
	\begin{align*}
		\Expnorm{\mt - \nFxk}
		\leq&\  
		\varsym \sqrt{ \sum_{\tau = 1}^\ii \betaprod {\pare{\tau + 1}} \ii^2 \pare{1-\iterationbeta \tau}^2}
		+ \Lz \sum_{\tau = 2}^\ii \iterationeta \tau \betaprod \tau \ii \iterationC \tau
		 + \Lo \sum_{\tau = 2}^\ii \iterationeta \tau \betaprod \tau \ii \iterationD \tau \Expnorm{\nF \pare{\iterationx \tau}},
	\end{align*}
	where $\prodb a b$ denotes $\prod_{\ii = a}^b \betat$ and $\Ct, \Dt$ are defined as in \Cref{lem:analysis.descent_lemma}.
\end{lemma}

\begin{proof}
    This proof is motivated by \citet{nsgdm}, and similar arguments are carried by \citet{ImprovedAnalysis} and \citet{Junchi2022}. To simplify notation we first define
    \begin{align*}
        \mut &\coloneqq \mt - \nFxk,\\
        \gat &\coloneqq \gt - \nFxk,\\
        \at &\coloneqq 1- \betat, \\
        \betaprod a b &\coloneqq \prod_{t = a}^b \betat.
    \end{align*}
    Now let $i,j \in \N, i < j$ and calculate
    \begin{align}\label{eq:analysis.uncorrelated}
        \begin{split}
            \Exp{\gaj^\top \gai}
            &= \Exp{\Exp[\Fc_{j-1}]{\gaj^\top \gai}}\\
            &= \Exp{\Exp[\Fc_{j-1}]{\gaj}^\top \gai}\\
            &= 0,
        \end{split}
    \end{align}
    where we used that $\Exp[\Fc_{j-1}]{\gaj} = 0$ in the last equality. Next we define $\St \coloneqq \nF \pare \xtm - \nF \pare \xt$ and calculate
    \begin{align*}
        \mt 
        &= \betat \mtm + \pare{1-\betat}\gt \\
        &= \betat \pare{\nF \pare \xkm + \mutm} + \pare{1- \betat} \pare{\gat + \nFxk}\\
        &= \nFxk + \pare{1 - \betat} \gat + \betat \St + \betat \mutm.
    \end{align*}
    This yields
    \begin{align*}
        \mut
        &= \betaprod 2 \ii \iterationmu 1 
        + \sum_{\tau = 2}^\ii \betaprod {\pare{\tau + 1}} \ii \iterationa \tau \iterationga \tau
        + \sum_{\tau = 2}^\ii \betaprod \tau \ii \iterationS \tau
        =  \sum_{\tau = 1}^\ii \betaprod {\pare{\tau + 1}} \ii \iterationa \tau \iterationga \tau
        + \sum_{\tau = 2}^\ii \betaprod \tau \ii \iterationS \tau,
    \end{align*}
    where we used $\iterationbeta \fin = 0$ in the second inequality. Therefore
    \begin{align*}
        \Expnorm \mut
        &\leq \Expnorm{\sum_{\tau = 1}^\ii \betaprod {\pare{\tau + 1}} \ii \iterationa \tau \iterationga \tau}
        + \sum_{\tau = 2}^\ii \betaprod \tau \ii \Expnorm{\iterationS \tau}.
    \end{align*}
    To further concretize this upper bound, \eqref{eq:analysis.uncorrelated} firstly yields
    \begin{align*}
        \Expnorm{\sum_{\tau = 1}^\ii \betaprod {\pare{\tau + 1}} \ii \iterationa \tau \iterationga \tau}
        &\leq \sqrt{ \sum_{\tau = 1}^\ii \betaprod {\pare{\tau + 1}} \ii^2 \at^2 \varsym^2}.
    \end{align*}
    Secondly, \assumlzlo~implies
    \begin{align*}
        \norm \St
        &\leq \etat \pare{ \Ct \Lz+ \Dt  \Lo\norm \nFxk} 
    \end{align*}
    and hence
    \begin{align*}
        \sum_{\tau = 2}^\ii \betaprod \tau \ii \Expnorm{\iterationS \tau}
        &\leq 
        \Lz \sum_{\tau = 2}^\ii \iterationeta \tau \betaprod \tau \ii \iterationC \tau
        + \Lo \sum_{\tau = 2}^\ii \iterationeta \tau \betaprod \tau \ii \iterationD \tau \Expnorm{\nF \pare{\iterationx \tau}}.
    \end{align*}
    Putting these results together we get the claim.
\end{proof}

Now we are ready for the main result.

\begin{theorem}[\nsgdm~for \LzLo-smoothness]\label{thm:analysis.background_main_result}
	\newcommand{\descentUb}{F \pare{\iterationx \fin} - F^* + \constOne}
	Assume \assumlb, \assumlzlo~and \assumbounded. Let $\ssi > 0$ and define the parameters
	\begin{align*}
		\betat &\coloneqq 1 - \ii^{-\nicefrac 1 2}\\
		\etat &\coloneqq \ssi\ii^{-\nicefrac 3 4}.
	\end{align*}
	Then \nsgdm~with starting point $\iterationx \fin \in \R^d$ satisfies
	\begin{align*}
		\isum \frac {\etat} 2 \Expnorm \nFxk
		\leq&\ \Dz + \eta \varsym \pare{7 + 2\sqrt{2e^2} \log \pare \li} + \eta^2 \Lz \pare{45\frac{e^\Aa - 1}\Aa + 14 \log \pare \li}\\
		&\ + 21 \ssi^2 \Lz e^{48 \pare \Aa^2} + 6 \ssi e^{48 \pare \Aa^2} \norm{\nF \pare \xfin},
	\end{align*}
	where $\Dz \coloneqq F \pare{\xfin} - F^*$. Furthermore, if $\Lo \geq \nicefrac 1 {2 \ssi}$, the statement also holds when replacing $6 \ssi e^{48 \pare \Aa^2} \norm{\nF \pare \xfin}$ with $\frac {e^{48 \pare \Aa^2}} {\Lo} \norm{\nF \pare \xfin}$.
\end{theorem}

The main workhorse behind the following proof is \Cref{lem:technical.explicit}. It intuitively states that the quantities which emerge due to the nonconstant parameters behave (nearly) \emph{asymptotically the same} as constant stepsizes would.

\begin{proof}
	\newcommand{\dT}{\delta_\ii}
	To simplify notation we define
	\begin{align*}
		\betaprod a b \coloneqq \prod_{\tau = a}^b \iterationbeta \tau.
	\end{align*}
	We start the proof by combining \Cref{lem:analysis.descent_lemma} and \Cref{lem:analysis.mu_bound} to obtain
	\begin{align*}
		\isum \etat \Expnorm \nFxk
		\stackAlign{\ref{lem:analysis.descent_lemma}}{\leq}\,
		\Dz + \frac {\Lz} 2 \isum \etat^2 \Ct + \frac {\Lo} 2 \isum \etat^2 \Dt \Expnorm \nFxk
		+ 2 \isum \etat \Expnorm{\nFxk - \mt}\\
		\stackAlign{\ref{lem:analysis.mu_bound}}{\leq}\,
		\Dz + \frac {\Lz} 2 \isum \etat^2 \Ct + \frac {\Lo} 2 \isum \etat^2 \Dt \Expnorm \nFxk
		+ 2 \varsym \isum \etat \sqrt{\sum_{\tau = 1}^\ii \iterationa \tau^2 \pare{\betaprod{\pare{\tau + 1}}{\ii}}^2}\\
		&\ +2 \Lz \isum \etat \sum_{\tau = 2}^\ii \iterationeta \tau \iterationC \tau \betaprod \tau \ii
		+ 2 \Lo \isum \etat \sum_{\tau = 2}^\ii \iterationeta \tau \iterationD \tau \betaprod \tau \ii \Expnorm{\nF \pare{\iterationx \tau}}.
	\end{align*}
	Next, we use \Cref{lem:technical.explicit} $a$) and $b$) to bound all terms that are independent of the iterates $\xt$. This leaves us with
	\begin{align}\label{lzlo.infnsgd.eq:gen_main_thm_main_bound_1}
		\begin{split}
			\isum \etat \Expnorm \nFxk
			\leq&\ \Dz + \eta \varsym \pare{7 + 2\sqrt{2e^2} \log \pare \li} + \eta^2 \Lz \pare{45\frac{e^\Aa - 1}\Aa + 14 \log \pare \li}\\
			&\ + \underbrace{\frac {\Lo} 2 \isum \etat^2 \Dt \Expnorm \nFxk + 2\Lo \sum_{\tau = 2}^\li \etatau \pare{\sum_{\ii = \tau}^\li \etat \betaprod \tau \ii} \iterationD \tau \Expnorm{\nF \pare \xtau}}_{\eqqcolon (A)}, 
		\end{split}
	\end{align}
	where we rearranged the sums of the last term. We then focus on upper bounding $(A)$. Therefore we use \Cref{lem:technical.general} $ii)$ which yields
	\newcommand{\const}{M}
	\begin{align*}
		(A) 
		&\leq
		\isum \etat \Dt \pare{\frac {\Lo} 2 \etat + 2e^{2\pare{\sqrt 2 - 1}}\Lo \eta \ii^{-\of} } \Expnorm \nFxk
		\leq \isum \etat \Dt \pare{\const \Lo \ssi \ii^{- \nicefrac 1 4} } \Expnorm \nFxk,
	\end{align*}
	where $\const \coloneqq \frac 1 2 + 2 \exp\pare{2\sqrt 2 - 2} \leq 5.1$.
	In a setting with access to problem parameters, we could now set $\ssi \coloneqq \frac 1 {12 \Lo}$ and hence guarantee that $\const \ssi \Lo \ii^{-\of} \Dt \leq \frac 1 2$, which would complete the proof. In the parameter agnostic setting we have to wait until the stepsize decreased below this threshold. We therefore define the threshold $\tz \coloneqq \ceil{\pare{12\eta \Lo}^4}$ after which we again have  $\const\ssi \Lo \ii^{-\of} \Dt \leq \frac 1 2$. This is due to $\Dt \leq \iterationD \tz \leq \frac {12} {2\const}$ for $t \geq \tz$. We are therefore left with the task of controlling the sum in $(A)$ up to $\tz$, i.e.~$(B)$ in
	\begin{align}\label{lzlo.infnsgd.eq:gen_main_thm_A_bound_1}
		(A) 
		\leq \underbrace{\sum_{\ii = 1}^{\tz - 1} \etat \pare{\const \Lo \ssi \ii^{- \nicefrac 1 4} \Dt} \Expnorm \nFxk}_{(B)}
		+ \sum_{\ii = \tz}^{\li} \frac \etat 2 \Expnorm \nFxk.
	\end{align}	
	We start by upper bounding $\norm \nFxk$ using \assumlzlo. For $\dT \coloneqq \norm{\xt - \iterationx \fin} \leq 4 \eta \ii ^\of - 3 \ssi$ our smoothness assumption implies
	\begin{align*}
		\norm \nFxk 
		\leq \norm{\nF \pare \xfin} + \norm{\nFxk - \nF \pare \xfin} \leq e^{\Lo \dT} \Lz \dT + e^{\Lo \dT} \norm{\nF \pare \xfin}
	\end{align*}
	and plugging into $(B)$ yields
	\begin{align*}
		(B) 
		\leq \underbrace{\pare{\ssi \const \sum_{\ii = 2}^{\tz - 1} \Lo \etat \ii^{-\of} \dT \Dt e^{\Lo \dT}}}_{\eqqcolon (B1) }\Lz
		+\underbrace{ \pare{ \ssi  \const \sum_{\ii = 1}^{\tz - 1} \Lo \ssi \ii^{-1} \Dt e^{\Lo \dT}}}_{\eqqcolon (B2)} \norm{\nF \pare \xfin}.
	\end{align*}
	Now \Cref{lem:technical.explicit} $c)$ $i)$ allows us to upper bound $(B1)$ via
	\begin{align*}
		(B1) 
		&\leq \ssi^2 \const \Lz \pare{ \frac \Aa 2 e^{2\Aa} + 4 e^{- \frac 5 2 \Aa} \pare{e^{4 \Aa \pare{\tz - 1}^\of} - e^{4 \Aa}}}\\
		&\leq \ssi^2 \const \Lz \pare{\pare{\frac \Aa 2 - 4}e^{2 \Aa} + 4e^{4 \Aa \pare{\tz - 1}^\of}}\\
		&\leq \ssi^2 \const \Lz \pare{\pare{\frac \Aa 2 - 4}e^{2 \Aa} + 4e^{48 \pare \Aa^2}},
	\end{align*}
	where we used the definition of $\tz$ in the last inequality. Next we use that, for all $x \geq 0$, we have $(\nicefrac x 2 - 4) e^{4x} + e^{48x^2} \leq \frac {21} {4\const} e^{48x^2}$ and hence
	\begin{align*}
		(B1) \leq 21 \ssi^2 \Lz e^{48 \pare \Aa^2}.
	\end{align*}
	Using \Cref{lem:technical.explicit} $c)$ $ii)$ and the same technique as for $(B1)$ we obtain
	\begin{align*}
		(B2)
		&\leq \ssi \const \pare{\frac 3 2 \Aa e^{\nicefrac 5 3 \Aa} + e^{- \nicefrac 5 2 \Aa} \pare{e^{4 \Aa \pare{\tz - 1}^\of} - e^{4\Aa}}}\\
		&\leq \ssi \const \pare{\pare{\frac{3}{2}\Aa - 1}e^{2\Aa} + e^{48\pare\Aa^2}}\\
		&\leq 6 \ssi e^{48 \pare \Aa^2} < \frac 3 {\Lo} e^{48 \pare \Aa^2}.
	\end{align*}
	We plug these results into \eqref{lzlo.infnsgd.eq:gen_main_thm_A_bound_1} to obtain
	\begin{align*}
		(A) \leq 
		21 \ssi^2 \Lz e^{48 \pare \Aa^2} \Lz
		+ 6 \ssi e^{48 \pare \Aa^2} \norm{\nF \pare \xfin}
		+ \sum_{\ii = \tz}^{\li} \frac \etat 2 \Expnorm \nFxk
	\end{align*}
	and combing with \eqref{lzlo.infnsgd.eq:gen_main_thm_main_bound_1} yields
	\begin{align*}
		\frac 1 2 \isum \etat \Expnorm \nFxk
		\leq&\ \Dz + \eta \varsym \pare{7 + 2\sqrt{2e^2} \log \pare \li} + \eta^2 \Lz \pare{45\frac{e^\Aa - 1}\Aa + 14 \log \pare \li}\\
		&\ + 21 \ssi^2 \Lz e^{48 \pare \Aa^2} \Lz + 6 \ssi e^{48 \pare \Aa^2} \norm{\nF \pare \xfin}.
	\end{align*}
	This finishes the proof of the first statement.
	
	For the second statement assume $\Aa \geq \nicefrac 1 2$. In this case we apply \Cref{lem:technical.explicit} $c)$ $iii)$ and get
	\begin{align*}
		(B2)
		&\leq \ssi \const \pare{\frac 3 2 \Aa e^{\nicefrac 5 3 \Aa} + e^{-\nicefrac 5 2 \Aa} 
			\pare{2 \pare{\tz - 1}^{- \nicefrac 1 4} e^{4\Aa \pare{\tz-1}^{\nicefrac 1 4}} - e^{4\Aa} }}\\
		&\leq \ssi \const \pare{\pare{\frac 3 2 \Aa - 1}e^{2 \Aa} + \frac 1 {6 \Aa}e^{48 \pare \Aa^2}}\\
		&\leq \frac 1 {\Lo} e^{48 \pare \Aa^2}
	\end{align*}
	Proceeding as before yields the second claim.	
\end{proof}

By plugging in $\ssi = \nicefrac 1 7$ we now get the formal result of \Cref{thm:main_result}.

\begin{corollary}\label{cor:analysis.remove_stepsize}
	Assume \assumlb, \assumlzlo~and \assumbounded. Furthermore define the parameters $\betat \coloneqq 1 - \ii^{-\nicefrac 1 2}$ and $\etat \coloneqq \frac {\ii^{-\nicefrac 3 4}} 7$.
	Then \nsgdm~with starting point $\iterationx \fin \in \R^d$ satisfies
	\begin{align*}
		\frac 1 \li \isum \Expnorm \nFxk
		\leq&\ 
		\frac{\pare{14 + 96 \Lo e^{\Lo^2}}\Dz + \pare{14e^{\nicefrac {\Lo} 7} + 9 \log \pare \li + 2e^{\Lo^2}}\Lz} 
		{\li^\of} \\
		&\ + \frac{12e \log \pare \li \varsym  + 6e^{\Lo^2} \min \set{\frac {\Lz} {\Lo}, \sqrt{8 \Lz \Dz}}}{\li^\of},
	\end{align*}
	where $\Dz \coloneqq F \pare{\xfin} - F^*$ is the initialization gap. Furthermore, if $\Lo \geq \nicefrac 7 2$, we get the following improved dependence on $\Lo$:
	\begin{align*}
		\frac 1 \li \isum \Expnorm \nFxk
		\leq&\ 
		\frac{126 e^{\Lo^2} \Dz + 12e\log \pare \li \varsym + \pare{4 e^{\Lo^2} + 9 \log \pare \li} \Lz}
		{\li^\of}.
	\end{align*}
\end{corollary}

\begin{proof}
	Plugging the choice of $\ssi = \frac 1 7$ into \Cref{thm:analysis.background_main_result} and using that $\log \pare \li \geq 1$ yields
	\begin{align*}
		\frac {\ssi} {2} \isum \ii^{- \nicefrac 3 4} \Expnorm{\nFxk}
		\leq&\ 
		\Dz + {6e} \ssi \log \pare \li \varsym + 6 \ssi e^{\Lo^2} \norm{\nF \pare \xfin}
		 + \ssi \Lz \pare{7 e^{\nicefrac {\Lo} 7} + \frac {30} 7 \log \pare \li + 2 e^{\Lo^2}}.
	\end{align*}
	Next, from the proof of \Cref{lem:app.gradient_bound}, we get that 
	\begin{align*}
		\norm{\nF \pare{\xfin}} \leq 8 \Lo \Dz + \min \set{\frac {\Lz} {\Lo}, \sqrt{8 \Lz \Dz}}
	\end{align*}
	and hence, by noting that $\frac 1 \li \isum \Expnorm{\nFxk} \leq \li^{-\of} \isum \ii^{- \nicefrac 3 4} \Expnorm{\nFxk}$ we obtain
	\begin{align*}
		\frac 1 {\li^{\nicefrac 3 4}} \isum \Expnorm{\nFxk}
		\leq&\ \pare{14 + 96 \Lo e^{\Lo^2}}\Dz + 12e \log \pare \li \varsym
		+ \pare{14e^{\nicefrac {\Lo} 7} + 9 \log \pare \li + 2e^{\Lo^2}}\Lz\\
		&\ + 6e^{\Lo^2} \min \set{\frac {\Lz} {\Lo}, \sqrt{8 \Lz \Dz}}
	\end{align*}
	and hence proved the first claim.
	
	For the second claim assume $\Lo \geq \nicefrac 7 2$. We now can use the second statement in \Cref{thm:analysis.background_main_result} to get
	\begin{align*}
		\frac 1 {\li^{\nicefrac 3 4}} \isum \Expnorm{\nFxk}
		\leq&\ \pare{14 + 112 e^{\Lo^2}}\Dz + 12e \log \pare \li \varsym
		+ \pare{14e^{\nicefrac {\Lo} 7} + 9 \log \pare \li + 2e^{\Lo^2}}\Lz\\
		&\ + \frac {e^{\Lo^2}} {\Lo} \min \set{\frac {\Lz} {\Lo}, \sqrt{8 \Lz \Dz}}\\
		\leq&\ 126 e^{\Lo^2} \Dz + 12e\log \pare \li \varsym + \pare{4 e^{\Lo^2} + 9 \log \pare \li} \Lz,
	\end{align*}
	where we used that $14e^{\nicefrac {\Lo} 7} + \Lo^{-2} e^{\Lo^2} \leq e^{\Lo^2}$ for $\Lo \geq \nicefrac 7 2$.
\end{proof}

Finally we provide the formal statement of \Cref{cor:non_agnostic_nsgdm}.

\begin{corollary}[Non parameter-agnostic \nsgdm]\label{cor:analysis.non_agnostic}
	\newcommand{\descentUb}{F \pare{\iterationx \fin} - F^* + \constOne}
	Assume \assumlb, \assumlzlo~and \assumbounded. Furthermore define the parameters $\betat \coloneqq 1 - \ii^{-\nicefrac 1 2}$ and $\etat \coloneqq \frac {\ii^{-\nicefrac 3 4}} {12\Lo}$.
	Then \nsgdm~with starting point $\iterationx \fin \in \R^d$ satisfies
	\begin{align*}
		\frac 1 \li \isum \Expnorm \nFxk
		\leq \frac{24 \Lo \Dz + \pare{14 + 4 \sqrt{2e^2}\log \pare \li} \varsym + \pare{10 + 4 \log \pare \li} \frac {\Lz} {\Lo}}{\li^{\nicefrac 1 4}}.
	\end{align*}
	where $\Dz \coloneqq F \pare{\xfin} - F^*$ is the initialization gap.
\end{corollary}

\begin{proof}
	Denote $\ssi \coloneqq \nicefrac 1 {12}$. By plugging our choice of $\etat$ into \eqref{lzlo.infnsgd.eq:gen_main_thm_main_bound_1} we obtain
	\begin{align*}
		\isum \frac 1 2 \etat \Expnorm \nFxk
		\leq&\ \Dz + \ssi \varsym \pare{7 + 2\sqrt{2e^2} \log \pare \li} + \eta^2 \Lz \pare{55 +21 \log \pare \li}
	\end{align*}
	and by using the same arguments as in the proof of \Cref{thm:main_result} we get
	\begin{align*}
		\frac 1 \li \isum \Expnorm \nFxk
		\leq \frac{\frac{2\Dz}{\ssi} + \pare{14 + 4 \sqrt{2e^2}\log \pare \li} \varsym + \ssi \pare{110 + 42 \log \pare \li} \Lz}{\li^{\nicefrac 1 4}}.
	\end{align*}
\end{proof}

\subsubsection{Deterministic Setting} 
\label{sec:app.missing_proofs.upper_bounds.deterministic_setting}

In this subsection, we provide the result for \backtrGD.

\begin{proof}[Proof of \Cref{thm:linesearch_result}]
	\newcommand{\ub}{ u\pare{\Lz, \Lo, \Dz}}
	By \Cref{lem:app.gradient_bound} we have that $\norm{ \nF \pare x} \leq \max \set{8 \Lo \pare{F(x) - F^*}, \frac \Lz \Lo}$. Since \backtrGD~is a descent algorithm, we get that $\norm \nFxk \leq \max \set{8 \Lo \Dz, \frac \Lz \Lo}\eqqcolon \ub$ for all $\ii \in \N$. Now let $x\in \R^d$ be an iterate of \backtrGD~and $\eta \leq \frac 1 {\Lo }$. Then \Cref{lem:app.lzlo_property} implies
	\begin{align*}
		F \pare{x - \eta \nFx} 
		&\leq F(x) - \eta \nsq \nFx + \eta^2\pare{2\Lz + \pare{e-1}\Lo \norm \nFx} \nsq \nFx \\
		&\leq F(x) - \eta \nsq \nFx + \eta^2 \pare{2\Lz + \pare{e-1} u \pare{\Lz, \Lo, \Dz}} \nsq \nFx \\
		&= F(x) - \eta \pare{1 - \eta L}\nsq \nFx,
	\end{align*}
	where $L \coloneqq 2\Lz + \pare{e-1}\Lo \ub$. In particular we have that $F \pare{x - \eta \nF \pare x} \leq F(x) - \eta \beta \nsq \nFx$ whenever $\eta \leq \frac{1 - \beta} L$. This allows us to lower bound our stepsizes by $\etat > \frac{\gamma \pare{1-\beta}} L$. As in the $L$-smooth setting, the definition of $\xtp$ now yields
	\begin{align*}
		\frac \beta \li \isum \etat \nsq \nFxk
		&\leq \frac {\Dz} {\li}
	\end{align*}
	and thus
	\begin{align*}
		\frac 1 \li \isum \nsq \nFxk
		&\leq \frac{L\Dz}{\beta \gamma \pare{1-\beta} \li}.
	\end{align*}
	This finishes the proof.
\end{proof}

\subsection{Proofs for Parameter-Agnostic Lower Bounds}\label{sec:app.missing_proofs.lower_bounds}

In this section, we provide the formal proofs for \Cref{sec:lower_bounds}.

\begin{algorithm2e}
	\caption{General Normalized Momentum Method}
	\label{alg:general_momentum_method}
	\DontPrintSemicolon
	\SetKwInOut{Input}{Input}
	\Input{Starting point $\iterationx \fin \in \R^d$, stepsize $\ssi > 0$, power $\al > 0$}
	$\iterationm 0 \gets 0$\;
	\For{$\iterationIndex = 1, 2, \,\hdots$}{
		Independently sample $\xit$ from the distribution of $\xi$.\;
		$\gk \gets \nf \pare{\xk, \xit}$\;
		Choose $\mk \in \cone \pare{\ls{\iterationg \fin}{\gt}} \setminus \set{0}$\;
		$\xkp \gets \xk - \frac \ssi {\ii^\al} \frac {\mk} {\norm{\mk}}$\;
	}
\end{algorithm2e}

\begin{proof}[Proof of \Cref{lem:momentumlb_functions}]
	Define $z_1 \coloneqq \frac 2 {\Lz}$ and the derivatives
	\begin{align*}
		p'(x) &\coloneqq \Lz x 		& p_1'(x) &\coloneqq p'(x) - 1, & p_2'(x) & \coloneqq p'(\ssi-x)-1,\\
		q'(x) &\coloneqq e^{\Lo x} 	& q_1'(x) &\coloneqq q'(x-z_1), & q_2'(x) & \coloneqq q'(\ssi-z_1-x).
	\end{align*}
	We now define the function $F$ via its derivative
	\begin{align*}
		F' \coloneqq 
		- \one_{(-\infty, 0)}
		+ \one_{[0, z_1)} p_1' 
		+ \one_{[z_1, \nicefrac \ssi 2)} q_1' 
		+ \one_{[\nicefrac \ssi 2, \ssi - z_1)} q_2' 
		+ \one_{[\ssi - z_1, z_2)} p_2'
		- 2 \eps \one_{[z_2, z_3)}
		+ \one_{[z_3, z_4)} h
	\end{align*}
	where $z_2 \coloneqq \ssi - z_1 + \frac{1+2\eps}{\Lz} \leq  \ssi$ and $z_3, z_4$ and $h$ will be determined later. Then $F(x) \coloneqq \Dz + \int_0^x F' d \lambda$ (see \Cref{fig:lower_bound_F}) satisfies
	\begin{align*}
		F(x) = &\ 
		\Dz 
		+ \frac 2 {\Lo} \pare{e^{\Lo \pare{\nicefrac \ssi 2 - z_1}} - 1} \one_{[\nicefrac \ssi 2, \infty)}(x)
		- \one_{(-\infty, 0)}(x) \cdot x\\
		&\ + \one_{[0, z_1)}(x) \cdot \pare{\frac{\Lz} 2 x^2 - x}
		+ \one_{[z_1, \nicefrac \ssi 2)}(x) \cdot \frac 1 {\Lo} \pare{e^{\Lo \pare{x - z_1}} - 1}\\
		&\ - \one_{[\nicefrac \ssi 2, \ssi - z_1)} \pare{x} \frac 1 {\Lo} \pare{e^{\Lo \pare{\ssi - z_1 - x}} - 1} 
		+ \one_{[\ssi - z_1, z_2)} \pare{x} \pare{x - \ssi - z_1 -\frac{\Lz \pare{x - 1 - z_1}^2}{2}}\\
		&\ - 2 \eps \one_{[z_2, z_3)}\pare{x}
		+ \one_{[z_3, z_4)}\pare{x} h\pare{x}
	\end{align*}
	and in particular
	\begin{align*}
		F \pare \ssi \geq \Dz + \frac 2 {\Lo} \pare{e^{\Lo \pare{\nicefrac \ssi 2 - z_1}} - 1}.
	\end{align*}
	By our choice of $\Lz$ we get that $\frac \ssi 2 - z_1 \geq \frac \ssi 4$ which implies $F \pare \ssi \geq \Dz + \frac 2 {\Lo} \pare{e^{\frac \Aa 4} - 1} \eqqcolon C$. We have
	\begin{align*}
		\iterationx T 
		= \ssi \sum_{\ii = 1}^{\li - 1} t^{-\al} 
		\leq \ssi \pare{1 + \frac 1 {1-\al} \pare{\pare{\li - 1}^{1-\al} - 1}}
		\leq \frac \ssi {1-\al} \li^{1 - \al}
	\end{align*}
	and hence
	\begin{align*}
		F \pare{\iterationx \li}
		\geq C - 2\eps \pare{\iterationx \li - \ssi}
		\geq 2 \ssi \eps + C - \frac{2\ssi \eps}{1-\al} \li^{1-\al}.
	\end{align*}
	Since
	\begin{align*}
		\li \leq \pare{\frac{1-\al}{2}}^{\frac 1 {1 - \al}} \pare{\frac{\Dz} \ssi + \frac 2 \Aa \pare{e^{\frac \Aa 4} - 1}}^{\frac 1 {1 - \al}}\eps^{-\frac 1 {1 - \al}}
	\end{align*}
	now implies that $F \pare{\iterationx \li} \geq 2 \ssi \eps$, the gradient of $F$ at $\iterationx \li$ is still $2\eps > \eps$ and we have not yet reached an $\eps$-stationary point. Finally we are left with the task of flattening $F$ out while making sure it never attains negative values and is still \LzLo-smooth. Therefore set $z_3 \coloneqq \ssi + \frac{C}{2 \eps}$ and $z_4 \coloneqq z_3 + \frac{2\eps}{\Lz}$. Now let $h(x) \coloneqq p'(x-z_3) - 2\eps$ and note that this achieves the exact goal we were aiming for.
	
	The only thing left to do, is to show that $F$ is indeed \LzLo-smooth. It is clear that $F$ is \LzLo-smooth on each of the subintervals $(-\infty, 0), ..., [z_3, z_4), [z_4, \infty)$. The claim hence follows from the upcoming \Cref{lem:lb.helper_lem}.
\end{proof}

\begin{lemma}\label{lem:lb.helper_lem}
	Let $I \subseteq \R$ be an interval, $a \in I$ and set $I_- \coloneqq \set{x \in I \mid x \leq a}$, $I_+ \coloneqq \set{x \in I \mid x \geq a}$. Further Let $f \colon \R \to \R$ be continuously differentiable and suppose that $f$ satisfied the inequality from \Cref{def:lzlo_smoothness} on $I_+$ and $I_-$. Then the inequality is also satisfied on $I$, i.e.~it also holds for $x \in I_-, y \in I_+$.
\end{lemma}

\begin{proof}
	W.l.o.g.~let $x \in I_-, y \in I_+$ and set $c \coloneqq \Lo \norm{x-y}$. Furthermore set $c_1 \coloneqq \Lo \norm{x-a}, c_2 \coloneqq \Lo \norm{a-y}$ and calculate
	\begin{align}\label{eq:lb.helper_lem.1}
		\begin{split}
			\norm{\nf (x) - \nf(y)}
			=&\ \norm{\nf(x) - \nf(a) + \nf(a) - \nf(y)}\\
			\leq&\ \Lz \pare{\norm{x-a}\Az \pare{c_1} + \norm{a-y}\Az \pare{c_2}}\\
			&\ + \Lo \norm{x-a} \Ao \pare{c_1} \norm{\nf(x)} + \Lo \norm{a-y} \Ao \pare{c_2} \norm{\nf(a)}.
		\end{split}
	\end{align}
	Next, since $a \in I_-$, we get that
	\begin{align*}
		\norm{\nf(a)} \leq \Lz \norm{x-a} \Az \pare{c_1} + e^{c_1} \norm{\nf(x)}
	\end{align*}
	and hence
	\begin{align*}
		\Lo \norm{a-y} \Ao \pare{c_2} \norm{\nF (a)}
		\leq&\ 
		\Lz \Lo \norm{a-y} \Ao \pare{c_2} \norm{x-a} \Az \pare{c_1}
		 + \Lo \norm{a-y} \Ao \pare{c_2} e^{c_1} \norm{\nf(x)}\\
		=&\ \Lz \pare{e^{c_2} - 1}\norm{x-a}\Az\pare{c_1}
		 + \Lo \norm{a-y} \Ao \pare{c_2} e^{c_1} \norm{\nf(x)}
	\end{align*}
	We now plug this result into \eqref{eq:lb.helper_lem.1} and rearrange to obtain
	\begin{align}\label{eq:lb.helper_lem.2}
		\begin{split}
			\norm{\nf (x) - \nf(y)}
			&\leq 
			\Lz \pare{e^{c_2} \norm{x-a}\Az\pare{c_1} + \norm{a-x}\Az\pare{c_2}}\\
			&~~~+ \Lo \norm{\nf(x)} \pare{\norm{x-a}\Ao \pare{c_1} + \norm{a-y} \Ao \pare{c_2}e^{c_1}}.
		\end{split}
	\end{align}
	Now we focus on the second term, involving $\Lo \norm{\nf(x)}$. Therefore we calculate
	\begin{align*}
		\begin{split}
			&\norm{x-a}\Ao \pare{c_1} + \norm{a-y} \Ao \pare{c_2}e^{c_1}\\
			&= \frac{e^{\Lo \norm{x-a}} - 1}{\Lo} + \frac{e^{\Lo \norm{x-y}} - e^{\Lo \norm{x-a}}}{\Lo}\\
			&= \Ao \pare{c} \norm{x-y}.
		\end{split}
	\end{align*}
	Next we focus on the first term in \eqref{eq:lb.helper_lem.2}, which corresponds to the $\Lz$-dependence. Calculating yields
	\begin{align*}
		\begin{split}
			&e^{c_2} \norm{x-a}\Az\pare{c_1} + \norm{a-y}\Az\pare{c_2}\\
			&= \norm{x-a}e^{\Lo \norm{a-y}} + \norm{x-a}e^{\Lo \norm{x-y}} - \frac{e^{\Lo \norm{x-y}} - e^{\Lo \norm{a-y}}}{\Lo}\\
			&~~~+ \norm{a-y} + \norm{a-y}e^{\Lo \norm{a-y}} - \frac{e^{\Lo \norm{a-y}} - 1}{\Lo}\\
			&= \norm{a-y} + \norm{x-y}e^{\Lo \norm{a-y}} + \norm{x-a}e^{\Lo \norm{x-y}} - \frac{e^{\Lo \norm{x-y}} - 1}{\Lo}\\
			&\leq \norm{x-y} + \norm{x-y}e^{\Lo \norm{x-y}} - \frac{e^{\Lo \norm{x-y}} - 1}{\Lo} = \Az \pare{\Lo \norm{x-y}} \norm{x-y}.
		\end{split}
	\end{align*}
	In the last inequality we used that for all $a, b, \Lo \geq 0$ the following inequality holds: $b + (a+b)e^{\Lo b} + be^{\Lo \pare{a+b}} \leq a + b + \pare{a+b}^{\Lo \pare{a+b}}$. This follows by taking partial derivatives with respect to $\Lo$. Finally we plug everything into \eqref{eq:lb.helper_lem.2} and obtain
	\begin{align*}
		\norm{\nf \pare{x} - \nf \pare y}
		\leq \pare{\Az \pare{c} \Lz + \Ao \pare{c} \Lo \norm{\nf \pare x}} \norm{x-a}.
	\end{align*}
	This finishes the proof.
\end{proof}

\newpage
\section{Additional Discussions on Parameter-Agnostic Lower Bounds}
\label{sec:app.additional_discussion_lower_bounds}
In this section, we provide further discussion on the notion of parameter-agnostic lower bounds. Additionally, we highlight the difference between \Cref{def:lb} and the condition in \Cref{prop:sufficient_lb}.

The section is organised as follows: We start by introducing the necessary notation, assumptions, and definitions in \Cref{sec:app.additional_discussion_lower_bounds.prelims}. Subsequently, in \Cref{sec:app.additional_discussion_lower_bounds.another_view}, we present an alternative way to motivate our definition of parameter-agnostic lower bounds. This alternative perspective allows for a more intuitive distinctions between \Cref{def:lb} and the condition in \Cref{prop:sufficient_lb}, as discussed in \Cref{remark:app.another_view.difference}.

\subsection{Preliminaries}
\label{sec:app.additional_discussion_lower_bounds.prelims}

\begin{notation}
	Throughout this section, let $\Theta \subseteq \R^k$ denote a parameter space that is unbounded in each dimension, i.e.\ there exists a sequence $\theta^{(n)} \in \Theta$ such that $\theta^{(n)}_i \to \infty$ for all $i \in [k]$.
	Additionally, let $\FkThet = \set{\Fcthet \colon \theta \in \Theta}$ be a parameterized family of function spaces.
\end{notation}

For simplicity, we furthermore assume that all algorithms satisfy $\xfin = 0$. If this is not the case, we can apply $A$ to the shifted function $\tilde f(x) = f \pare{x - \xfin}$. For the scope of this section, we therefore restrict $\Acdet$ to deterministic algorithms that use $\xfin = 0$.

Lastly, we introduce a multivariate $\Oc$-notation. While the extension of the $\Oc$-notation to a multivariate setting comes with technical complexities, as noted by \cite{MultivariateO}, the straightforward extension is sufficient for our purposes.

\begin{definition}[Multivariate $\Oc$-Notation]\label{def:multivariate_O_notation}
	Consider a function $h \colon (0, \infty) \times \Theta \to [0, \infty]$. We employ the following definitions:
	\begin{enumerate}[label=\roman*)]
		\item The multivariate $\Oc$ is given by the set
		\begin{align*}
			\Oc \pare{h} \coloneqq \set{f \colon (0, \infty) \times \Theta \to [0, \infty] \mid \exs \eps_0, \theta_0, K > 0\, \fas \eps \in (0, \eps_0], \theta \geq \theta_0
				\colon f \pare{\eps, \theta} \leq K h \pare{\eps, \theta}}.
		\end{align*}
		\item Analogously, the multivariate $o$ is defined as the set
		\begin{align*}
			o \pare{h} \coloneqq \set{f \colon (0, \infty) \times \Theta \to [0, \infty] 
				\mid \fas \kappa > 0\,\exs \eps_0, \theta_0 >0 \, \fas \eps \in (0, \eps_0], \theta \geq \theta_0 \colon
				f \pare{\eps, \theta} \leq \kappa h \pare{\eps, \theta}}.
		\end{align*}
	\end{enumerate}
	Here $\theta \geq C$ is to be understood component-wise. We also adopt standard $\Oc$-notation $f \pare{\eps, \theta} = \Oc\pare{h \pare{\eps, \theta}}$, $\pare{\eps \to 0, \theta \to \infty}$ to indicate $f \in \Oc \pare{h}$. Analogously, we use $f \pare{\eps, \theta} = o\pare{h \pare{\eps, \theta}}, \ \pare{\eps \to 0, \theta \to \infty}$ to signify $f \in \Oc \pare{h}$.
\end{definition}

\subsection{Another Point of View}
\label{sec:app.additional_discussion_lower_bounds.another_view}

In this section, we re-examine the definition of parameter-agnostic lower bounds through the lens of order theory. This perspective serves two purposes. Firstly, it enables us to formally compare the performance of two parameter-agnostic algorithms. Secondly, it better highlights the differences between \Cref{def:lb} and \Cref{prop:sufficient_lb}.

To start off, we address the question of how to compare different parameter-agnostic algorithms to determine which one is ``better''. To this end, we first introduce the concept of parameter-agnostic complexity of an algorithm, which maps each combination of $\theta$ and $\eps$ to the corresponding worst-case performance.

\begin{definition}[Parameter-Agnostic Complexity of an Algorithm]\label{def:app.another_view.agnostic_complexity}
	For any $A \in \Acdet$ we call $h_A \colon (0, \infty) \times \Theta \to [1, \infty]$,
	\begin{align*}
		h_A \pare{\eps, \theta} \coloneqq \sup_{f \in \Fcthet} \cMeasure\pare{A, f}
	\end{align*} 
	the \emph{parameter-agnostic complexity for $A$ on $\FkThet$}. Here $\cMeasure\pare{A, f} = \inf\set{\ii \in \Ngeq \mid \norm{\nf \atxk} \leq \eps}$ denotes the number of iterations required for $A$ to reach an $\eps$-stationary point of $f$.
\end{definition}

To illustrate this notion, let us consider the example of Gradient Descent with constant stepsizes applied to $L$-smooth functions.

\begin{example}\label{ex:app.another_view.agnostic_compl_of_constant_GD}
	Let $\set{A_\ssi \colon \ssi > 0} = \Ac \subseteq \Acdet$ be the set of Gradient Descent algorithms with constant stepsizes $\ssi > 0$ and $\xfin = 0$. Furthermore, for each $L \geq 0$ and $\Dz \geq 0$, let $\FcLDz$ denote the set of all $L$-smooth functions with initialization gap $F\pare{0} - \inf_{x\in\R^d} F(x) \leq \Dz$, and $\Theta = [0, \infty)^2$. For each $A_\ssi \in \Ac$ we will now calculate the parameter-agnostic complexity on $\FkThet$. Firstly, for $\ssi < \nicefrac 2 L$ it is well known that 
	\begin{align*}
		h_{A_\ssi} \pare{\eps, L, \Dz} \leq \ceil*{\frac{\Dz}{\ssi \pare{\frac{L \ssi}2 - 1}}\eps^{-2}}.
	\end{align*}
	On the other hand, if $\ssi \geq \nicefrac 2 L$, we can construct the function $F(x) = \frac L 2 \pare{x+\sqrt 2 \frac \eps L}^2$ that is $L$-smooth and on which $A_\ssi$ will not converge. Hence we get that
	\begin{align}\label{eq:app.another_view.example1.constructedF}
		\cMeasure\pare{A_\ssi, F} = \infty.
	\end{align}
	Now note that $F$ belongs to $\FcLDz$ for all $\Dz \geq \nicefrac {\eps^2} L$.  Therefore \eqref{eq:app.another_view.example1.constructedF} implies that for all such $\Dz$ and $\ssi \geq \nicefrac 2 L$ we have $h_{A_\ssi}\pare{\eps, L, \Dz} = \infty$. In particular, as $L, \Dz \to \infty$ and $\eps \to 0$ we get that $h_{A_\ssi}\pare{\eps, L, \Dz} = \infty$.
\end{example}

Now that we have established a measure for the parameter-agnostic complexity of an individual algorithm, the next logical step is to consider how to compare two algorithms to determine which one is ``better''.  We argue that in general algorithms are considered better than others, if they have a preferable behaviour as problems get harder. We therefore introduce the following (pre-)order for parameter-agnostic complexities.

\begin{definition}[Ordering Complexities]\label{def:app.another_view.preorder}
	Let $\mathscr{C} = \set{f \colon (0, \infty) \times \Theta \to [1, \infty]}$ denote the set of all possible complexities. Then we define the relation $\preceq$ on $\mathscr{C}$ as
	\begin{align*}
		h_1 \preceq h_2 \Leftrightarrow h_1\pare{\eps, \theta} = \Oc \pare{h_2\pare{\eps, \theta}}, \qquad \pare{\eps \to 0, \theta \to \infty}
	\end{align*}
	where $\Oc$ denotes the multivariate $\Oc$-notation (see \Cref{def:multivariate_O_notation}).
\end{definition}

This definition paves the way for comparing the parameter-agnostic complexities of different algorithms. We say that a (parameter-agnostic) algorithm $A$ is at least as good as algorithm $B$, if $h_A \preceq h_B$. This observation naturally leads to the following definition.

\begin{definition}[Naïve Parameter-Agnostic Lower Bound]\label{def:app.another_view.weak_lb}
	Let $\Ac \subseteq \Acdet$ be an algorithm class and $g \colon (0, \infty) \times \Theta  \to [1, \infty]$. Then we call $g$ \emph{weak parameter-agnostic lower bound for $\Ac$ on $\FkThet$}, if
	\begin{align}\label{eq:app.another_view.weak_lb}
		\fas A \in \Ac \colon g \preceq h_A.
	\end{align}
\end{definition}

When comparing the definition of $\preceq$ with the assumption in \Cref{prop:sufficient_lb}, we can observe that \eqref{eq:app.another_view.weak_lb} is equivalent to the assumption stated in the proposition. Therefore, discussing the difference between \Cref{prop:sufficient_lb} and \Cref{def:lb} boils down to understanding how \Cref{def:lb} and \Cref{def:app.another_view.weak_lb} differ.

Though the concept of a weak parameter-agnostic lower bound is intuitive and straightforward, its limitations become evident when examined more closely. The following example highlights this issue.

\begin{example}\label{ex:app.another_view.weak_problem}
	Consider $\Ac = \set{A_1, A_2} \subseteq \Acdet$ and let $\FcLDz, \FkThet$ be defined as in \Cref{ex:app.another_view.agnostic_compl_of_constant_GD}. Suppose the parameter-agnostic complexities of $A_1$ and $A_2$ are given by
	\begin{align*}
		h_{A_1} \pare{\eps, L, \Dz} &= \frac{\Dz + e^L}{\eps^2},\\
		h_{A_2} \pare{\eps, L, \Dz} &= \frac{e^{\Dz} + L}{\eps^2}.
	\end{align*}
 	The best possible weak parameter-agnostic lower bound for $\Ac$ on $\FkThet$ is then given by $g \pare{\eps, L, \Dz} = \frac{\Dz + L}{\eps^2}$. However, this lower bound fails to capture the fact that all algorithms in $\Ac$ suffer from an exponential dependence on at least one parameter.
\end{example}

Motivated by this shortcoming of weak parameter-agnostic lower bounds, we instead chose \Cref{def:lb} for our notion of parameter-agnostic lower bounds. In our current setting, \Cref{def:lb} can be rephrased as follows.

\begin{proposition}\label{def:app.another_view.lb}
	Let $\Ac \subseteq \Acdet$ be an algorithm class and $g \colon (0, \infty) \times \Theta \to [1, \infty]$. Then $g$ is a parameter-agnostic lower bound of $\Ac$ on $\FkThet$ as defined in \Cref{def:lb} if and only if
	\begin{align}\label{eq:app.another_view.lb}
		\nexists A \in \Ac \colon h_A \prec g.
	\end{align}
	Here we define $h_1 \prec h_2$ if $h_1 \pare{\eps, \theta} = o \pare{h_2 \pare{\eps, \theta}}$ for $\eps \to 0, \theta \to \infty$ (see \Cref{def:multivariate_O_notation}).
\end{proposition}

Specifically, \eqref{eq:app.another_view.lb} ensures that no algorithm $A$ in the class $\Ac$ can have a parameter-agnostic complexity $h_A$ that is ``better'' (in the little-$o$ sense) than the proposed lower bound $g$.

Let us revisit \Cref{ex:app.another_view.weak_problem} to see how this definition fixes the previously discussed issue.

\begin{example}
	Consider the same setting as in \Cref{ex:app.another_view.weak_problem} and define $g_1 \pare{\eps, L, \Dz} \coloneqq \frac{\Dz + e^L}{\eps^2}, g_2 \pare{\eps, L, \Dz} \coloneqq \frac{e^{\Dz} + L}{\eps^2}$. Then both, $g_1$ and $g_2$ are parameter-agnostic lower bounds of $\Ac$ on $\FkThet$, while neither of them is a weak parameter-agnostic lower bound. This notion of lower bound does hence capture the fact, that there is exponential dependence in at least one variable.
\end{example}

This demonstrates the utility of employing the definition in \Cref{def:app.another_view.lb} over weak parameter-agnostic lower bounds. The more nuanced criterion allows for a better representation of the complexities from algorithms in $\Ac$. The following remark delves deeper into this distinction.

\begin{remark}\label{remark:app.another_view.difference}
	The main difference between \Cref{def:lb} and \Cref{def:app.another_view.weak_lb} (and therefore \Cref{prop:sufficient_lb}) is how they handle incomparable algorithms, i.e.~algorithms for which neither $h_A \preceq h_B$ nor $h_B \preceq h_A$. \Cref{def:app.another_view.weak_lb} enforces $a)$ that $g$ is comparable with all complexities and $b)$ that $g$ must be at least as good as all complexities. \Cref{def:lb} on the other hand only requires that complexities which are comparable with $g$ must not be strictly better than $g$.
	
	When focusing on parameters, the difference can be characterized as follows: A weak parameter-agnostic lower bound guarantees that there does not exist an algorithm in $\Ac$, that has a better dependence in \emph{any single parameter}. Parameter-agnostic lower bounds on the other hand guarantee, that there does not exist an algorithm $A$ which has better dependencies \emph{in all parameters}.

	From an order-theoretic standpoint, the difference is nearly the same as the difference between lower bounds and minimal elements. The only small difference is that we do not force $g$ to be in the set of complexities $\set{h_A \colon A \in \Ac}$.
\end{remark}

Finally we show that every weak parameter-agnostic lower bound is also a parameter-agnostic lower bound, as claimed by \Cref{prop:sufficient_lb}.

\begin{lemma}[Rephrased \Cref{prop:sufficient_lb}]\label{def:app.another_view.implication}
	Let $\Ac \subseteq \Acdet$ be an algorithm class and $g \colon (0, \infty) \times \Theta \to [1, \infty]$. If $g$ is a weak parameter-agnostic lower bound of $\Ac$ on $\FkThet$, then $g$ is also a parameter-agnostic lower bound of $\Ac$ on $\FkThet$.
\end{lemma}

\begin{proof}
	Let us first first recall the logical statements behind the two version of lower bounds. Firstly, \eqref{eq:app.another_view.weak_lb} can be rewritten to
	\begin{align}\label{eq:app.another_view.weak_lb.logic}
		\fas A \in \Ac \, \exs \eps_0, \theta_0, K > 0\, \fas \eps \in (0, \eps_0], \theta \geq \theta_0 \colon 
		g \pare{\eps, \theta} \leq K h_A \pare{\eps, \theta}.
	\end{align}
	Secondly, \eqref{eq:app.another_view.lb} corresponds to
	\begin{align}\label{eq:app.another_view.lb.logic}
		\fas A \in \Ac \, \exs \kappa > 0\, \fas \eps_0', \theta_0' > 0 \, \exs \eps \in (0, \eps_0'], \theta \geq \theta_0' \colon
		h_A \pare{\eps, \theta} \geq \kappa g \pare{\eps, \theta}.	
	\end{align}
	Now the proof is straightforward. Suppose $g$ satisfies \eqref{eq:app.another_view.weak_lb.logic} and let $A \in \Ac$.
	Choose $\kappa \coloneqq \nicefrac 1 K$ and let $\eps_0', \theta_0' > 0$ be arbitrary. Lastly define $\eps \coloneqq \min \set{\eps_0, \eps_0'}$ and choose $\theta \in \Theta$ such that $\theta \geq \theta_0$ and $\theta \geq \theta_0'$. This is possible due to our unboundedness assumption on $\Theta$ (see \Cref{sec:app.additional_discussion_lower_bounds.prelims}). Since $\eps \in (0, \eps_0]$ and $\theta \geq \theta_0$ we get that
	\begin{align*}
		g \pare{\eps, \theta} \leq K h_A \pare{\eps, \theta} = \frac 1 \kappa h_A \pare{\eps, \theta}.
	\end{align*}
	This completes the proof.
\end{proof}

\end{document}